\documentclass[leqno]{article}

\usepackage{latexsym}

\newtheorem{definition}{Definition}
\newtheorem{theorem}{Theorem}
\newtheorem{proposition}{Proposition}
\newtheorem{corollary}{Corollary}

\newcommand{\Prf}{Proof}
\newenvironment{proof}[1][\Prf]{
  \trivlist
  \item[\hskip\labelsep  
    \indent \textbf{#1}]%
}{%
$\hfill\Box$
}

\makeatletter
   \renewcommand{\theequation}{%
   \thesection.\arabic{equation}}
   \@addtoreset{equation}{section}
\makeatother

\usepackage{graphicx,color}
\usepackage[T1]{fontenc}
\usepackage{textcomp}
\newcommand{\sq}{\textquotesingle}
\newcommand{\vc}[1]{\mbox{\boldmath{$#1$}}}

\newcommand{\LB}{\left\{}
\newcommand{\RB}{\right\}}
\newcommand{\LA}{\left(}
\newcommand{\RA}{\right)}
\newcommand{\ip}[1]{\LA #1 \RA}
\newcommand{\TL}{\tilde}
\newcommand{\tvc}[1]{\TL{\mbox{\boldmath{$#1$}}}}
\newcommand{\tA}{\TL{A}}

\newcommand{\tU}{\tilde{U}}
\newcommand{\T}{{\rm T}}
\newcommand{\PO}{M}
\newcommand{\Pinv}{\PO^{-1}}
\newcommand{\Pinvt}{\PO^{-\T}}
\newcommand{\SH}{\sharp}
\newcommand{\NON}{\nonumber}
\newcommand{\UB}{\underline{ }}
\newcommand{\pR}{R}
\newcommand{\pP}{P}
\newcommand{\palpha}{\talpha}
\newcommand{\pbeta}{\tbeta}
\newcommand{\pZ}{\TL{\lambda}}
\newcommand{\PCONV}{=}
\newcommand{\rmL}{{\rm L}}
\newcommand{\rmW}{{\rm W}}
\newcommand{\rmR}{{\rm R}}

\def\L({\left(}
\def\R){\right)}
\def\i{k}
\def\j{k}
\def\P{M}
\def\KSP2{{\mathcal K}_{k+1}}
\def\KSP{${\mathcal K}_{k+1}$}

\def\NON{\nonumber}
\def\SP{\;}  % thick space
\def\UL{\underline}
\def\UB{\underline{ }}
\newcounter{alg}
\newenvironment{indention}[1]%
{\par\begingroup\addtolength{\leftskip}{#1}}%
{\par\endgroup}
\def\UL{\underline}
\def\UB{\underline{ }}

\newcommand{\ALGWIDTHA}{8em}
\newcommand{\ALGWIDTHB}{6em}

\def\FIGSIZE{0.6}
\newtheorem{remark}{\it Remark}
\newcommand{\Def}{Definition}
\newcommand{\Dfn}{\Def}

\newcommand{\Thm}{Theorem}

\newcommand{\Prop}{Proposition}

\newcommand{\Rem}{{\it Remark}}

\newcommand{\Alg}{Algorithm}
\newcommand{\Ref}{Reference}
\newcommand{\Sec}{section}

\newcommand{\App}{Appendix}
\newcommand{\Fig}{Figure}

\newcommand{\reeqno}[1]{\renewcommand{\theequation}{#1}}

\def\PBiCG{PBiCG}

\newcommand{\GRAPHAlpBet}{.}
\def\GRAPHDIR{.}
\newcommand{\GRAPHisrv}{.}
\def\GRAPHDIRTWO{.}

\title{
Structure of the polynomials in
preconditioned BiCG algorithms
and the switching direction of preconditioned systems
}

\author{
Shoji Itoh%
\thanks{Department of Engineering Science, Faculty of Engineering,
      Osaka Electro-Communication University. 
}
$\;\;$
and
$\;$
Masaaki Sugihara% \footnote{Deceased 5 January 2019}%
\thanks{Department of Physics and Mathematics,
     College of Science and Engineering, Aoyama Gakuin University.
(Deceased 5 January 2019)
}
}

\date{}

\begin{document}
\maketitle

\begin{abstract}
We present
a theorem that defines the direction of a preconditioned system
for the bi-conjugate gradient (BiCG) method.
The theorem is able to be extended
to a variety of preconditioned bi-Lanczos-type methods.
We analyze and compare the polynomial structures of
four preconditioned BiCG algorithms, in this paper.
Finally,
we show that the direction of a preconditioned system is switched
by construction and by the settings of the initial shadow residual vector.
\end{abstract}

%%%%%%%%%%%%%%%%%%%%%%%%%%%%%%%%%%%%%%%%%%%%%%%%%%%%%%%%%%%%%%%%%%%%%%%%%%%%%%%%%%%%
\section{Introduction}
\label{sec:intro}

The bi-Lanczos-type methods are based on the 
bi-conjugate gradient (BiCG) method \cite{fletcher1976,lanczos1952}
and solve the system of linear equations
\begin{eqnarray}
A\vc{x} &=& \vc{b},
\label{eqn:linear}
\end{eqnarray}
where $A$ is a large, sparse coefficient matrix
of size $n\times n$,
$\vc{x}$ is the solution vector,
and $\vc{b}$ is the right-hand side (RHS) vector.
Bi-Lanczos-type methods are a kind of Krylov subspace method,
and they
assume the existence of a dual system:
\begin{eqnarray}
A^\T\vc{x}^\SH &=& \vc{b}^\SH;
\label{eqn:shadow}
\end{eqnarray}
(\ref{eqn:shadow}) will be referred to as the ``shadow system''.
In general,
the degree $k$ of the Krylov subspace generated
by $A$ and $\vc{r}_0$ is displayed as
${\cal K}_k\L(A,\vc{r}_0\R)$
$={\rm span}\LB\vc{r}_0, A\vc{r}_0, A^2\vc{r}_0, \cdots, \right.$
 $\left. A^{k-1}\vc{r}_0\RB$,
where
$\vc{r}_0$ is the initial residual vector $\vc{r}_0 = \vc{b}-A\vc{x}_0$, for
an initial guess to the solution $\vc{x}_0$.
The Krylov subspace ${\cal K}_k\L(A,\vc{r}_0\R)$ generated by the $k$-th iteration
forms the structure of
$\vc{x}_k \in \vc{x}_0 + {\cal K}_k\L(A,\vc{r}_0\R)$,
where $\vc{x}_k$ is the approximate solution vector
(or simply the ``solution vector'').

In general,
with a preconditioned Krylov subspace method,
there are some different algorithms
depending on the preconditioning conversion.
The structure of the approximate solution
is often different for different algorithms,
and the performance of a given algorithm
may differ substantially from those of other algorithms \cite{itoh2015a, itoh2019a}.
In particular,
preconditioned bi-Lanczos-type algorithms construct dual systems,
and so their analysis is more complex.

The conjugate gradient squared (CGS) method \cite{sonneveld1989}
is one of the bi-Lanczos-type methods,
and an improved preconditioned CGS (improved PCGS) algorithm
has been proposed \cite{itoh2015a}.
In a previous study \cite{itoh2019a}, we compared
the structures of the recurrence formula of the solution vectors of four PCGS algorithms,
including the improved PCGS.
In this paper,
we analyze the structures on the polynomials of
the preconditioned BiCG (PBiCG) algorithms that
correspond to those analyzed in our previous study \cite{itoh2019a}.
Furthermore, in
~\cite{itoh2019a}, we also discussed the construction
of the initial shadow residual vector (ISRV)
in terms of the direction of the preconditioned system;
we further analyze this topic in this paper.

In this paper, when we refer to a
{\it preconditioned algorithm}, we mean one involving a
preconditioning operator $\PO$ or a preconditioning matrix,
and by {\it preconditioned system}, we mean
one that has been converted by some operator(s) based on $\PO$.
These terms never indicate
{\it the algorithm for the preconditioning operation itself},
such as incomplete LU decomposition or the approximate inverse.
For example,
for a preconditioned system,
the original linear system (\ref{eqn:linear}) becomes
\begin{eqnarray}
&&
\tA\tvc{x} = \tvc{b},
\label{eqn:plinear}
\\
&&
\tA     \PCONV \Pinv_L A \Pinv_R, \;\;
\tvc{x} \PCONV \PO_R\vc{x}, \;\;
\tvc{b} \PCONV \Pinv_L\vc{b},
\label{eqn:plinear_conv}
\end{eqnarray}
with the preconditioner $\PO = \PO_L\PO_R$ ($\PO \approx A$).
In this paper,
the matrix and the vector in the preconditioned system
are indicated by a tilde ($\;\TL{ }\;$).
However,
the conversions in (\ref{eqn:plinear}) and (\ref{eqn:plinear_conv}) are
not implemented directly;
rather, we construct the preconditioned algorithm
that is equivalent to solving (\ref{eqn:plinear}).

This paper is organized as follows.
In section~\ref{sec:pbicg_alg}, we analyze 
various PBiCG algorithms
in terms on their polynomial structures,
and we clarify the details of the PCGS algorithms discussed in ~\cite{itoh2019a}.
In section~\ref{sec:direction_preconditioning},
we present a theorem that defines the direction of a preconditioned system
for the BiCG method.
We analyze the mechanism
that switches the direction of a preconditioned system for the BiCG method,
and we provide the details for some instances that show that, 
depending on the construction and setting of the ISRV,
the BiCG method may be transformed to another method
or the direction of the preconditioned system may not be determined.
In section~\ref{sec:numerical_experiments},
we present some numerical results that verify
the equivalence of the PBiCG and PCGS methods,
the properties of each of the four PBiCG algorithms discussed
in \Sec~\ref{sec:pbicg_alg},
the switching of the direction of a preconditioned system for the BiCG method,
and the resulting basic properties,
as discussed in \Sec~\ref{sec:direction_preconditioning}.
Our conclusions are presented in section~\ref{sec:conclusion}.

%%%%%%%%%%%%%%%%%%%%%%%%%%%%%%%%%%%%%%%%%%%%%%%%%%%%%%%%%%%%%%%%%%%%%%%%%%%%%%%%%%%%
\section{Analysis of various preconditioned BiCG algorithms}
\label{sec:pbicg_alg}

In this section, we consider
four different PBiCG algorithms,
these PBiCG algorithms correspond to four PCGS algorithms
as shown
in \Fig~\ref{fig:PCGS_ConvLlprecImproved};
these are the same ones discussed in \cite{itoh2019a}.

\Alg~\ref{alg:pbicg_simple} can be used to derive these four PBiCG algorithms.
\\

%%%%%%%%%%%%%%%%%%%%%%%%%%%%%%%%%%%%%%%%%%%%%%%%%%%%%%%%%%%%%
\begin{figure}[htbp]
\begin{center}
\resizebox*{\FIGSIZE\columnwidth}{!}{
\includegraphics*{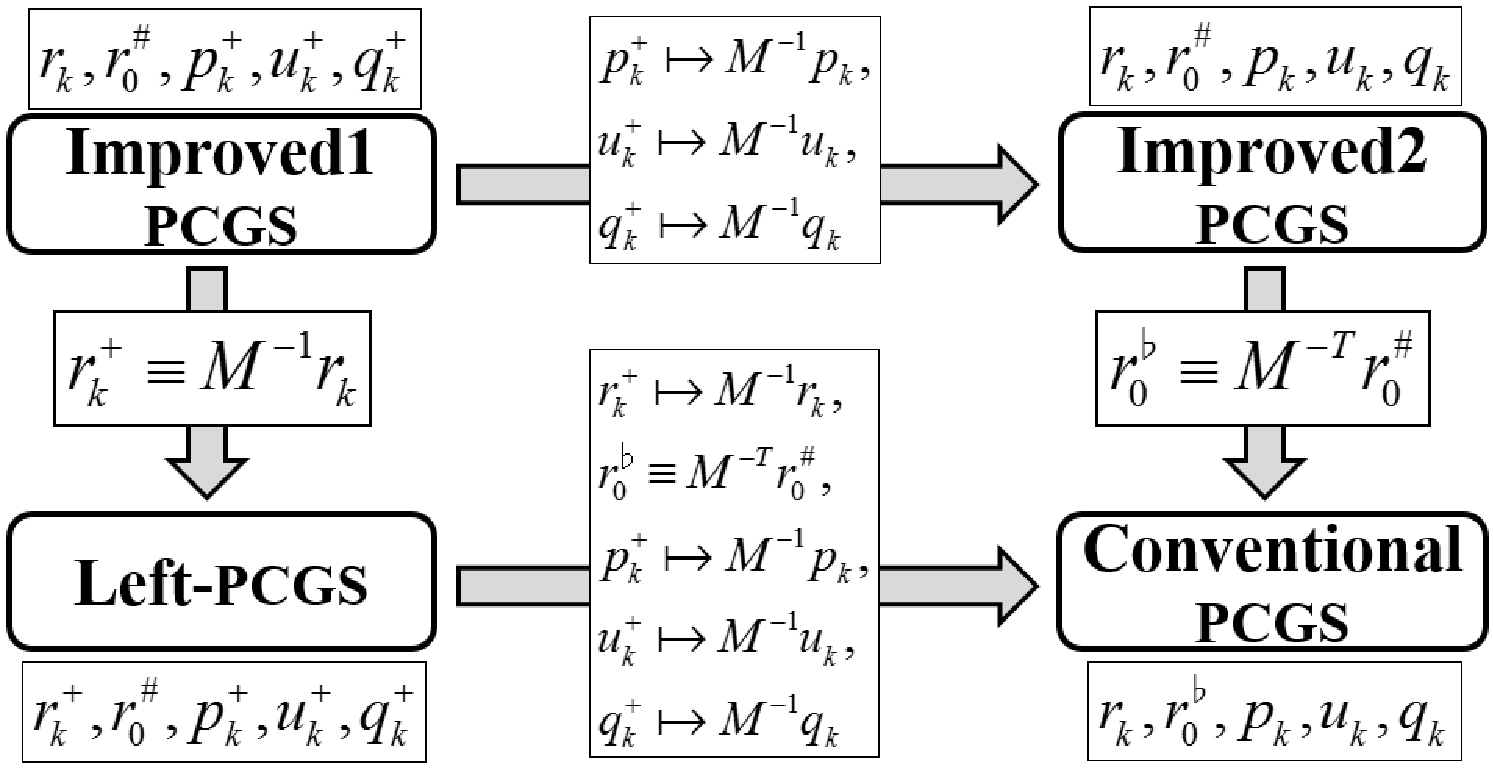}}
\caption{Relations between the four different PCGS
algorithms\cite{itoh2019a}.
$\mapsto$ : Splitting left vector to right members (preconditioner and vector),
$\equiv$ : Substituting left vector for right members.}
\label{fig:PCGS_ConvLlprecImproved}
\end{center}
\end{figure}
%%%%%%%%%%%%%%%%%%%%%%%%%%%%%%%%%%%%%%%%%%%%%%%%%%%%%%%%%%%%%

%%%%%%%%%%%%%%%%%%%%%%%%%%%%%%%%%%%%%%%%%%%%%%%%%%%%%%%%%%%%%
%\stepcounter{alg}
\refstepcounter{alg}
\label{alg:pbicg_simple}
\noindent
{\bf Algorithm \thealg.
BiCG method for preconditioned system:}
\begin{indention}{\ALGWIDTHB}
%%%%%%%%%%%%%%%%%%%%%%%%%%%%%%%%%%%%%%%%%%%%%%%%%%%%%%%%%%%%%
\noindent
$\tvc{x}_0$ is an initial guess,
$\SP\tvc{r}_0= \tvc{b}-\tA\tvc{x}_0$,
set $\SP\beta^{\rm PBiCG}_{-1}=0$,
\\
$\LA \tvc{r}^\SH_0, \tvc{r}_0 \RA \neq 0$,
e.g.,
$\tvc{r}^\SH_0 = \tvc{r}_0, \SP$
 \\
For $k = 0, 1, 2, \cdots ,$ until convergence, Do:
\begin{eqnarray}
&& \tvc{p}_k = \tvc{r}_k + \beta^{\rm \PBiCG}_{k-1}\tvc{p}_{k-1}, \NON \\ %\SP\SP
&& \tvc{p}^\SH_k = \tvc{r}^\SH_k + \beta^{\rm \PBiCG}_{k-1}\tvc{p}^\SH_{k-1}, \NON \\
&& \alpha^{\rm \PBiCG}_k = \frac
            {\LA \tvc{r}^\SH_k, \tvc{r}_k \RA}
            {\LA \tvc{p}^\SH_k, \tA\tvc{p}_k \RA}, \NON
\\
&& \tvc{x}_{k+1} = \tvc{x}_k + \alpha^{\rm \PBiCG}_k \tvc{p}_k, \NON \\
&& \tvc{r}_{k+1} = \tvc{r}_k - \alpha^{\rm \PBiCG}_k \tA\tvc{p}_k, \NON \\ %\SP\SP
&& \tvc{r}^\SH_{k+1} = \tvc{r}^\SH_k - \alpha^{\rm \PBiCG}_k \tA^\T\tvc{p}^\SH_k, \NON \\
&& \beta^{\rm \PBiCG}_k = \frac
            {\LA \tvc{r}^\SH_{k+1}, \tvc{r}_{k+1} \RA}
            {\LA \tvc{r}^\SH_k,     \tvc{r}_k     \RA}, \NON
\end{eqnarray}
End Do
\\

\end{indention}
%%%%%%%%%%%%%%%%%%%%%%%%%%%%%%%%%%%%%%%%%%%%%%%%%%%%%%%%%%%%%

Any preconditioned algorithm can be derived by substituting
the matrix with the preconditioner for the matrix with the tilde
and
the vectors with the preconditioner for the vectors with the tilde.
Obviously,
\Alg~\ref{alg:pbicg_simple} without the preconditioning conversion
is the same as the BiCG method.
If $\tA$ is a symmetric positive definite (SPD) matrix and
$\tvc{r}^\SH_0=\tvc{r}_0$,
then \Alg~\ref{alg:pbicg_simple} is mathematically equivalent to
the conjugate gradient (CG) method \cite{hestenes1952}
for a preconditioned system.

We present the following general definition; however,
the PBiCG will also require
\Thm~\ref{thm:dir_prec_by_isrv}, which will be presented in section 3.

\begin{definition}
\label{dfn:dir_prec}
For the system and solution

\begin{equation}
\reeqno{\ref{eqn:plinear}\sq}
\hspace*{-148pt}
\tA\tvc{x} = \tvc{b},
\end{equation}
\begin{equation}
\reeqno{\ref{eqn:plinear_conv}\sq}
\tA     \PCONV \Pinv_L A \Pinv_R, \;\;
\tvc{x} \PCONV \PO_R\vc{x}, \;\;
\tvc{b} \PCONV \Pinv_L\vc{b},
\end{equation}
we define
the direction of a preconditioned system of linear equations as follows:
\begin{itemize}
\item {\it The two-sided preconditioned system:}
Equation (\ref{eqn:plinear_conv}\sq);
\item {\it The right-preconditioned system:}
$\PO_L = I$ and $\PO_R = \PO$ in  (\ref{eqn:plinear_conv}\sq);
\item {\it The left-preconditioned system:}
$\PO_L = \PO$ and $\PO_R = I$ in  (\ref{eqn:plinear_conv}\sq),
\end{itemize}
where
$\PO$ is the preconditioner $\PO=\PO_L\PO_R$ ($\PO \approx A$), and
$I$ is the identity matrix.

Other vectors in the solving method are not preconditioned.
The initial guess is given as $\vc{x}_0$,
and $\tvc{x}_0 = \PO_R \vc{x}_0$.
\end{definition}

\setcounter{equation}{0}

The recurrence relations of the BiCG for a preconditioned system are
\begin{eqnarray}
\pR_0(\pZ) &=& 1, \SP\SP \pP_0(\pZ)=1, %\NON
\label{eqn:prec_init_pol}
\\
\pR_k(\pZ) &=&
  \pR_{k-1}(\pZ) -\alpha^{\rm \PBiCG}_{k-1}\pZ \pP_{k-1}(\pZ), %\SP\SP \pR_0(\pZ)=1, %\NON
\label{eqn:prec_res_pol}
\\
\pP_k(\pZ) &=&
  \pR_k(\pZ) +\beta^{\rm \PBiCG}_{k-1} \pP_{k-1}(\pZ). %,\SP\SP \pP_0(\pZ)=1 %\NON
\label{eqn:prec_prob_pol}
\end{eqnarray}
$\pR_k(\pZ)$ is the degree $k$ of the residual polynomial,
and
$\pP_k(\pZ)$ is the degree $k$ of the probing direction polynomial,
that is,
\begin{eqnarray}
\tvc{r}_k &\PCONV&
\pR_k(\tA)\tvc{r}_0,
\label{eqn:bicg_conv_tilde_rL}
\\
\tvc{p}_k &\PCONV&
\pP_k(\tA)\tvc{r}_0 .
\label{eqn:bicg_conv_tilde_pL}
\end{eqnarray}

Further,
to the shadow for the preconditioned system $\tA^\T\tvc{x}^\SH = \tvc{b}^\SH$,
we have
\begin{eqnarray}
\tvc{r}^\SH_k &\PCONV&
\pR_k(\tA^\T)\tvc{r}^\SH_0 ,
\label{eqn:bicg_conv_tilde_rS}
\\
\tvc{p}^\SH_k &\PCONV&
\pP_k(\tA^\T)\tvc{r}^\SH_0 .
\label{eqn:bicg_conv_tilde_pS}
\end{eqnarray}

%%%%%%%%%%%%%%%%%%%%%%%%%%%%%%%%%%%%%%%%%%%%%%%%%%%%%%%%%%%%%
\begin{figure}[t]
\begin{center}
\resizebox*{\FIGSIZE\columnwidth}{!}{
\includegraphics*{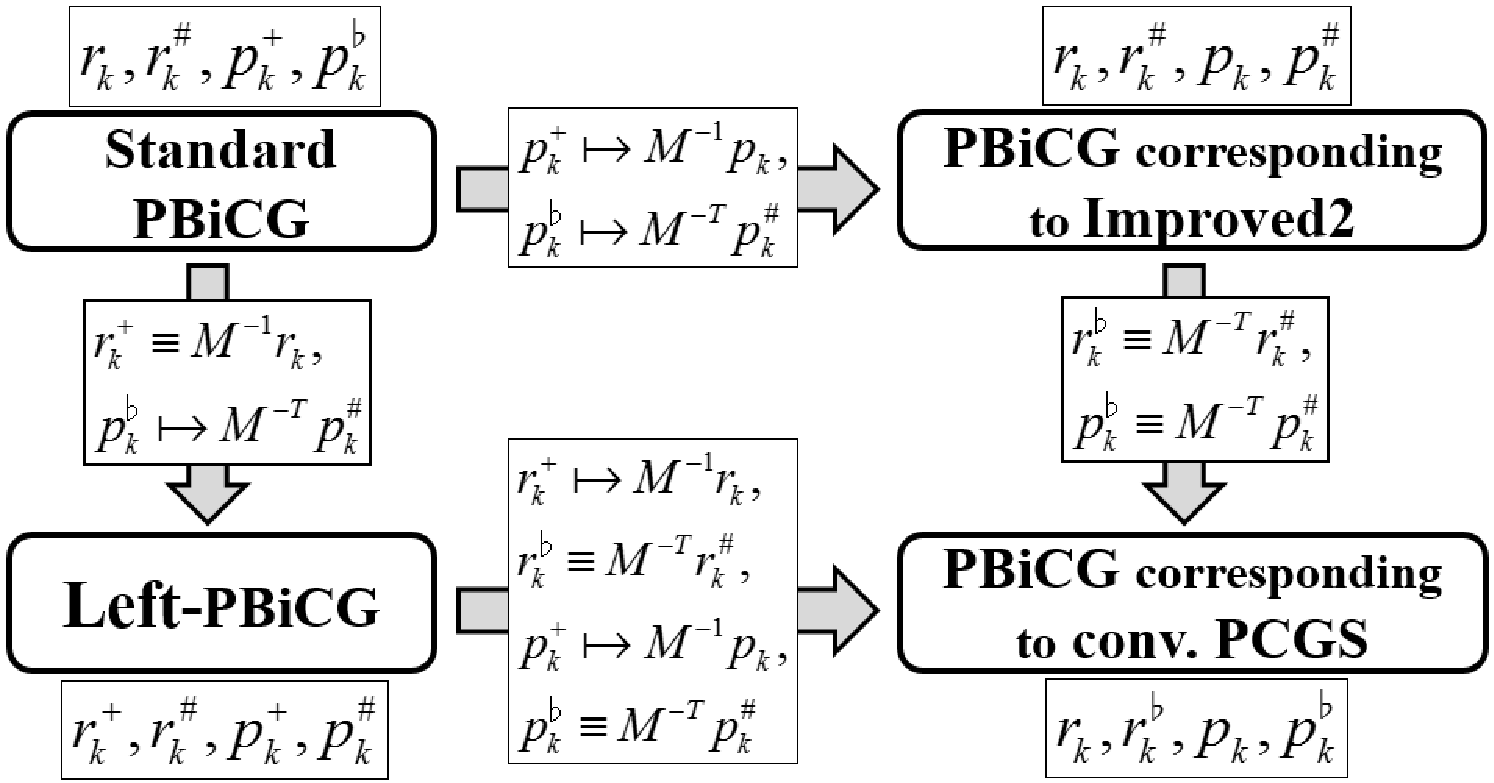}}
\caption{Relations between the four PBiCG algorithms that
correspond to the respective PCGS algorithms shown 
in \Fig~\ref{fig:PCGS_ConvLlprecImproved}.
$\mapsto$ : Splitting left vector to right members (preconditioner and vector),
$\equiv$ : Substituting left vector for right members.
}
\label{fig:PBiCG_ConvLlprecImproved}
\end{center}
\end{figure}
%%%%%%%%%%%%%%%%%%%%%%%%%%%%%%%%%%%%%%%%%%%%%%%%%%%%%%%%%%%%%
%
%
%
\begin{theorem}[Lanczos \cite{lanczos1952}, Fletcher \cite{fletcher1976},
Itoh and Sugihara \cite{itoh2015a}]
\label{thm:pbicg_biortho_biconj}
The BiCG method for a preconditioned system satisfies the following conditions:
\begin{eqnarray}
\ip{ \tvc{r}^\SH_i, \tvc{r}_j }
&=& 0
   \;\;\; (i\neq j), \;\;\; \mbox{ (biorthogonality),} % \NON
\label{eqn:pbiorthogonal}
% \\
\\
\ip{ \tvc{p}^\SH_i, \tA \tvc{p}_j }
&=& 0
   \;\;\; (i\neq j), \;\;\; \mbox{ (biconjugacy).} % \NON
\label{eqn:pbiconjugate}
% \\
\end{eqnarray}
\end{theorem}
\begin{proposition}
\label{prop:pbicg_biortho_biconj_alpha_beta}
The direction of a preconditioned system 
is determined by the operations of
$\alpha_k$ and $\beta_k$ in each PBiCG algorithm.
These intrinsic operations are based on
biorthogonality and biconjugacy.
\end{proposition}
\begin{theorem}
\label{thm:PBiCGforPCGS}
There exists a PBiCG algorithm that
corresponds to the preconditioning conversion defined by any given PCGS,
and the values of
$\alpha_k$ and $\beta_k$ will be equivalent to those of the PCGS.
\end{theorem}
\begin{proof}
See \cite{itoh2015a}.
\end{proof}

In particular,
\Ref~\cite{itoh2015a} explains the relations between
 $\alpha^{\rm PBiCG}_k$ and $\beta^{\rm PBiCG}_k$ of the standard PBiCG
 and
 $\alpha^{\rm PCGS}_k$ and $\beta^{\rm PCGS}_k$ of the improved PCGS.
In this paper,
we consider four PBiCG algorithms shown in \Fig~\ref{fig:PBiCG_ConvLlprecImproved},
and these correspond to the
four PCGS algorithms shown
in \Fig~\ref{fig:PCGS_ConvLlprecImproved}.

\subsection{PBiCG corresponding to conventional PCGS of the right system}
\label{ssec:pbicg_correspond_conv}

The PBiCG algorithm corresponding to the conventional PCGS
(the right-preconditioned system)
is derived by applying the following preconditioning conversion%
\footnote{%
In this case,
the shadow vectors of $\tvc{r}^\SH_k$ and $\tvc{p}^\SH_k$
are converted to $\PO^\T_L\vc{r}^\flat_k$ and $\PO^\T_L\vc{p}^\flat_k$,
but there is no problem with displaying
$\PO^\T_L\vc{r}^\SH_k$ and $\PO^\T_L\vc{p}^\SH_k$
in the notation of the algorithm.
However,
these internal structures are
 $\vc{r}^\flat_k \equiv \Pinvt\vc{r}^\SH_k$ and
 $\vc{p}^\flat_k \equiv \Pinvt\vc{p}^\SH_k$.
The details of this notation will be discussed
in sections~\ref{ssec:analysis_of_4pcgs} and \ref{sec:direction_preconditioning}.
The same applies to (\ref{eqn:pbicg_right_system_pcgs}).
}
to \Alg~\ref{alg:pbicg_simple}:
\begin{eqnarray}
&&
\tA       \PCONV \Pinv_L A \Pinv_R,  \SP
\tvc{x}_k \PCONV \P_R\vc{x}_k, \SP
\tvc{b}   \PCONV \Pinv_L\vc{b}, \SP %\NON
\label{eqn:pbicg_corresp_conv_pcgs}
\\
&&
\tvc{r}_k \PCONV \Pinv_L\vc{r}_k, \SP
\tvc{p}_k \PCONV \Pinv_L\vc{p}_k, \SP % \NON
% \\
\tvc{r}^\SH_k \PCONV \P^\T_L\vc{r}^\flat_k, \SP
\tvc{p}^\SH_k \PCONV \P^\T_L\vc{p}^\flat_k. \NON
\end{eqnarray}

\Alg~\ref{alg:pbicg_corresp_conv_pcgs} is presented below. %; its
\\

%
%%%%%%%%%%%%%%%%%%%%%%%%%%%%%%%%%%%%%%%%%%%%%%%%%%%%%%%%%%%%%
\refstepcounter{alg}
\label{alg:pbicg_corresp_conv_pcgs}
\noindent
{\bf Algorithm \thealg.
PBiCG algorithm corresponding to the conventional PCGS:}
\begin{indention}{\ALGWIDTHA}
%%%%%%%%%%%%%%%%%%%%%%%%%%%%%%%%%%%%%%%%%%%%%%%%%%%%%%%%%%%%%
\noindent
$\vc{x}_0$ is an initial guess,
$\SP\vc{r}_0= \vc{b}-A\vc{x}_0, \SP\SP$
set $\SP\beta_{-1}=0$,
\\
$\LA\TL{\vc{r}}^\SH_0, \TL{\vc{r}}_0\RA
 \PCONV \LA\vc{r}^\flat_0, \vc{r}_0\RA
 \neq 0$,
e.g.,
$\vc{r}^\flat_0 = \vc{r}_0, \SP$
\\
For $k = 0, 1, 2, \cdots ,$ until convergence, Do:
\begin{eqnarray}
&& \vc{p}_k = \vc{r}_k + \beta_{k-1}\vc{p}_{k-1},\SP\SP \NON \\
&& \vc{p}^\flat_k = \vc{r}^\flat_k + \beta_{k-1}\vc{p}^\flat_{k-1}, \NON
\\
&& \alpha_k = \frac
            {\LA\vc{r}^\flat_k, \vc{r}_k\RA}
            {\LA\vc{p}^\flat_k,  A\Pinv \vc{p}_k\RA}, %\NON
\label{eqn:alpha_rightPBiCG}
\\
&& \vc{x}_{k+1} = \vc{x}_k + \alpha_k\Pinv\vc{p}_k, \NON \\
&& \vc{r}_{k+1} = \vc{r}_k - \alpha_kA\Pinv\vc{p}_k, \SP\SP \NON \\
&& \vc{r}^\flat_{k+1} = \vc{r}^\flat_k - \alpha_k\Pinvt A^\T\vc{p}^\flat_k, \NON
\\
&& \beta_k = \frac
            {\LA\vc{r}^\flat_{k+1}, \vc{r}_{k+1}\RA}
            {\LA\vc{r}^\flat_k, \vc{r}_k\RA}, \NON
\end{eqnarray}
End Do
\\

\end{indention}
%%%%%%%%%%%%%%%%%%%%%%%%%%%%%%%%%%%%%%%%%%%%%%%%%%%%%%%%%%%%%

The stopping criterion is
\begin{eqnarray}
&&\frac{\|\vc{r}_{k+1}\|}{\|\vc{b}\|} \leq \varepsilon .
\label{eqn:conv_judge}
\end{eqnarray}

This algorithm can also be derived by the following conversion:
\begin{eqnarray}
&&
\tA \PCONV A \Pinv,  \SP
\tvc{x}_k  \PCONV \P\vc{x}_k, \SP
\tvc{b}    \PCONV \vc{b}, \SP % \NON
\label{eqn:pbicg_right_system_pcgs}
\\
&&
\tvc{r}_k \PCONV \vc{r}_k, \SP
\tvc{p}_k \PCONV \vc{p}_k, \SP % \NON
% \\
\tvc{r}^\SH_k \PCONV \vc{r}^\flat_k, \SP
\tvc{p}^\SH_k \PCONV \vc{p}^\flat_k.  \NON
\end{eqnarray}

This is the same as using $\PO_L=I$ and $\PO_R=\PO $
in (\ref{eqn:pbicg_corresp_conv_pcgs}).
Note that
this is the same as preconditioning to obtain
$\tA$, $\tvc{x}_k$, and $\tvc{b}$, 
but not converting the other vectors; thus, it is the
right-preconditioned system.

Now,
we convert $\tA$ and $\tvc{r}_0$ using (\ref{eqn:pbicg_corresp_conv_pcgs})
in order to obtain the polynomial representations of
(\ref{eqn:bicg_conv_tilde_rL}) and (\ref{eqn:bicg_conv_tilde_pL})
as $\tvc{r}_k$ and $\tvc{p}_k$, respectively:
\begin{eqnarray}
\tvc{r}_k % \PCONV \Pinv_L\vc{r}_k
              &=& \pR^\rmR_k(\tA)\tvc{r}_0
% \\
               =  \Pinv_L\pR^\rmR_k(A\Pinv)\vc{r}_0, \NON
\label{eqn:plinear_pol_r_conv}
\\
\tvc{p}_k % \PCONV \Pinv_L\vc{p}_k
              &=& \pP^\rmR_k(\tA)\tvc{r}_0
               =  \Pinv_L \pP^\rmR_k(A\Pinv)\vc{r}_0. \NON
\label{eqn:plinear_pol_p_conv}
\end{eqnarray}

We have denoted
these polynomials with a superscript ``\rmR'' %
\footnote{
In a similar manner,
we will use
``\rmL'' to indicate left-preconditioned system
  and
``\rmW'' to indicated two-sided preconditioned system
(see section~\ref{sec:direction_preconditioning}).
},
to indicate that
\Alg~\ref{alg:pbicg_corresp_conv_pcgs}, which
corresponds to the conventional PCGS method,
is a right-preconditioned system \cite{itoh2019a}.
The ISRV is set as $\vc{r}^\flat_0=\vc{r}_0$ in this algorithm.

Furthermore,
we use (\ref{eqn:pbicg_corresp_conv_pcgs}) to convert $\tvc{r}_k$ and $\tvc{p}_k$:
\begin{eqnarray}
\vc{r}_k &=& \pR^\rmR_k(A\Pinv)\vc{r}_0, %\NON
\label{eqn:plinear_pol_r}
\\
\vc{p}_k &=& \pP^\rmR_k(A\Pinv)\vc{r}_0 . %\NON
\label{eqn:plinear_pol_p}
\end{eqnarray}

The shadow system is also treated in a similar manner
using (\ref{eqn:pbicg_corresp_conv_pcgs}):
\begin{eqnarray}
\tvc{r}^\SH_k % \PCONV \P^\T_L\vc{r}^\flat_k
  &=&    \pR^\rmR_k(\tA^\T)\tvc{r}^\SH_0
% \\
   =    \P^\T_L\pR^\rmR_k(\Pinvt A^\T)\vc{r}^\flat_0 , \NON
\label{eqn:pshadow_pol_r_conv}
\\
\tvc{p}^\SH_k %\PCONV \P^\T_L\vc{p}^\flat_k
  &=&   \pP^\rmR_k(\tA^\T)\tvc{r}^\SH_0
% \\
   =    \P^\T_L\pP^\rmR_k(\Pinvt A^\T)\vc{r}^\flat_0 . \NON
\label{eqn:conv_pshadow_pol_p}
\end{eqnarray}
Finally, we have
\begin{eqnarray}
\vc{r}^\flat_k &=&
\pR^\rmR_k(\Pinvt A^\T)\vc{r}^\flat_0 , %\NON
\label{eqn:pshadow_pol_r}
\\
\vc{p}^\flat_k &=&
\pP^\rmR_k(\Pinvt A^\T)\vc{r}^\flat_0 . %\NON
\label{eqn:pshadow_pol_p}
\end{eqnarray}

We note that
(\ref{eqn:plinear_pol_r}), (\ref{eqn:plinear_pol_p}),
(\ref{eqn:pshadow_pol_r}), and (\ref{eqn:pshadow_pol_p})
can also be obtained using (\ref{eqn:pbicg_right_system_pcgs}).

The structures of
 biorthogonality (\ref{eqn:pbiorthogonal})
and
 biconjugacy (\ref{eqn:pbiconjugate})
are as follows:
\begin{eqnarray}
\LA\tvc{r}^\SH_i, \tvc{r}_j\RA
   &\PCONV&   \LA \P^\T_L\vc{r}^\flat_i, \Pinv_L\vc{r}_j\RA
    =    \LA \vc{r}^\flat_i, \vc{r}_j\RA %\NON
\label{eqn:conv_pbiorthogonal1}
\\
   &=&   \LA \pR^\rmR_i(\Pinvt A^\T)\vc{r}^\flat_0,\SP \pR^\rmR_j(A \Pinv)\vc{r}_0\RA
, \NON
\\
\LA\tvc{p}^\SH_i, \tA\tvc{p}_j\RA
   &\PCONV&   \LA \P^\T_L \vc{p}^\flat_i,  (\Pinv_L A \Pinv_R)(\Pinv_L\vc{p}_j)\RA
    =    \LA \vc{p}^\flat_i, (A \Pinv)\vc{p}_j\RA %\NON
\label{eqn:conv_pbiconjugate}
\\
   &=&   \LA \pP^\rmR_i(\Pinvt A^\T)\vc{r}^\flat_0,\SP (A\Pinv) \pP^\rmR_j(A\Pinv)\vc{r}_0\RA
. \NON
\end{eqnarray}

In \Alg~\ref{alg:pbicg_corresp_conv_pcgs},
the structures of 
$\vc{r}^\flat_k = \pR^\rmR_k(\Pinvt A^\T)\vc{r}^\flat_0$
and
$\vc{p}^\flat_k = \pP^\rmR_k(\Pinvt A^\T)\vc{r}^\flat_0$
are fixed,
and
their coefficient matrices are fixed as $\Pinvt A^\T$,
because
the ISRV is $\vc{r}^\flat_0$,
and
$\pR^\rmR_k(\Pinvt A^\T)\vc{r}^\flat_0$
cannot be transformed into
$\Pinvt \pR^\rmR_k(A^\T \Pinvt)\vc{r}^\SH_0 $.
Therefore,
the coefficient matrix of their linear system is $A\Pinv$,
so
 $\P\vc{x}_{k+1} = \P\vc{x}_k + \alpha^\rmR_k P^\rmR_k(A\Pinv)\vc{r}_0$
is structured,
where $\alpha^\rmR_k$ means (\ref{eqn:alpha_rightPBiCG});
and
 \Alg~\ref{alg:pbicg_corresp_conv_pcgs} is confirmed
 to correspond to the right-preconditioned system.

\subsection{PBiCG corresponding to the left system PCGS (Left-PBiCG)}
\label{ssec:pbicg_Llprec}

The left-PBiCG algorithm corresponding to the left-PCGS
can be derived by using the following preconditioning conversion%
\footnote{%
The notation $\vc{r}^+_k$ is important and will be discussed
in section \ref{ssec:analysis_of_4pcgs},
but there is no problem with displaying
$\vc{r}_k$ in the notation of the algorithm.
However,
its internal structure is
$\vc{r}^+_k \equiv \Pinv\vc{r}_k$.
Note that this is also true for $\vc{p}^+_k$.
}
in \Alg~\ref{alg:pbicg_simple}:
\begin{eqnarray}
&&
\tA\PCONV \Pinv A,  \SP
\tvc{x}_k\PCONV\vc{x}_k, \SP
\tvc{b}\PCONV\Pinv\vc{b}, \SP %\NON
\label{eqn:pbicg_corresp_Llpcgs}
\\
&&
\tvc{r}_k \PCONV \vc{r}^+_k, \SP
\tvc{p}_k \PCONV \vc{p}^+_k, \SP % \NON
% \\
\tvc{r}^\SH_k \PCONV \vc{r}^\SH_k, \SP
\tvc{p}^\SH_k \PCONV \vc{p}^\SH_k.  \NON
\end{eqnarray}
% \\

%%%%%%%%%%%%%%%%%%%%%%%%%%%%%%%%%%%%%%%%%%%%%%%%%%%%%%%%%%%%%
\refstepcounter{alg}
\label{alg:pbicg_corresp_Llpcgs}
\noindent
{\bf Algorithm \thealg.
PBiCG algorithm corresponding to left-PCGS:}
\begin{indention}{\ALGWIDTHA}
%%%%%%%%%%%%%%%%%%%%%%%%%%%%%%%%%%%%%%%%%%%%%%%%%%%%%%%%%%%%%
\noindent
$\vc{x}_0$ is an initial guess,
$\SP\vc{r}^+_0= \Pinv\L(\vc{b}-A\vc{x}_0\R), \SP\SP$
set $\SP\beta_{-1}=0$,
\\
$\LA\TL{\vc{r}}^\SH_0, \TL{\vc{r}}_0\RA
 \PCONV \LA\vc{r}^\SH_0, \vc{r}^+_0\RA
 \neq 0$,
e.g.,
$\vc{r}^\sharp_0 = \vc{r}^+_0, \SP$
\\
For $k = 0, 1, 2, \cdots ,$ until convergence, Do:
\begin{eqnarray}
&& \vc{p}^+_k = \vc{r}^+_k + \beta_{k-1}\vc{p}^+_{k-1},\SP\SP \NON \\
&& \vc{p}^\SH_k = \vc{r}^\SH_k + \beta_{k-1}\vc{p}^\SH_{k-1}, \NON
\\
&& \alpha_k = \frac
            {\LA\vc{r}^\SH_k, \vc{r}^+_k\RA}
            {\LA\vc{p}^\SH_k, \Pinv A\vc{p}^+_k\RA}, %\NON
\label{eqn:alpha_leftPBiCG}
\\
&& \vc{x}_{k+1} = \vc{x}_k + \alpha_k\vc{p}^+_k, \NON
\\
&& \vc{r}^+_{k+1} = \vc{r}^+_k - \alpha_k\Pinv A\vc{p}^+_k, \SP\SP \NON \\
&& \vc{r}^\SH_{k+1} = \vc{r}^\SH_k - \alpha_kA^\T\Pinvt\vc{p}^\SH_k, \NON
\\
&& \beta_k = \frac
            {\LA\vc{r}^\SH_{k+1}, \vc{r}^+_{k+1}\RA}
            {\LA\vc{r}^\SH_k, \vc{r}^+_k\RA}, \NON
\end{eqnarray}
End Do
\\

\end{indention}
%%%%%%%%%%%%%%%%%%%%%%%%%%%%%%%%%%%%%%%%%%%%%%%%%%%%%%%%%%%%%

In this algorithm,
the stopping criterion is
\begin{eqnarray}
\frac{\|\vc{r}^+_{k+1}\|}{\|\Pinv\vc{b}\|} \leq \varepsilon .
\label{eqn:conv_judge2}
\end{eqnarray}

The polynomials of the linear system are converted as follows:
\begin{eqnarray}
&&\tvc{r}_k % \PCONV \vc{r}^+_k
               =  \pR^\rmL_k(\tA)\tvc{r}_0
               =  \pR^\rmL_k(\Pinv A)\vc{r}^+_0,
\label{eqn:plinear_pol_r_Llconv}
\\
&&\tvc{p}_k % \PCONV \vc{p}^+_k
               =  \pP^\rmL_k(\tA)\tvc{r}_0
               =  \pP^\rmL_k(\Pinv A)\vc{r}^+_0,
\label{eqn:plinear_pol_p_Llconv}
\end{eqnarray}
and
\begin{eqnarray}
\vc{r}^+_k &=& \pR^\rmL_k(\Pinv A)\vc{r}^+_0, \NON
\label{eqn:plinear_pol2_r_Llconv}
\\
\vc{p}^+_k &=& \pP^\rmL_k(\Pinv A)\vc{r}^+_0. \NON
\label{eqn:plinear_pol2_p_Llconv}
\end{eqnarray}
In the shadow system, we have
\begin{eqnarray}
&&\tvc{r}^\SH_k % \PCONV \vc{r}^\SH_k
              = \pR^\rmL_k(\tA^\T)\tvc{r}^\SH_0
              = \pR^\rmL_k(A^\T \Pinvt)\vc{r}^\SH_0, \NON
\label{eqn:pshadow_pol_r_Llconv}
\\
&&\tvc{p}^\SH_k % \PCONV \vc{p}^\SH_k
              = \pP^\rmL_k(\tA^\T)\tvc{r}^\SH_0
              = \pP^\rmL_k(A^\T \Pinvt)\vc{r}^\SH_0, \NON
\label{eqn:pshadow_pol_p_Llconv}
\end{eqnarray}
and
\begin{eqnarray}
\vc{r}^\SH_k &=& \pR^\rmL_k(A^\T\Pinvt)\vc{r}^\SH_0, \NON
\label{eqn:pshadow_pol2_r_Llconv}
\\
\vc{p}^\SH_k &=& \pP^\rmL_k(A^\T\Pinvt)\vc{r}^\SH_0. \NON
\label{eqn:pshadow_pol2_p_Llconv}
\end{eqnarray}

The structures of
 biorthogonality % (\ref{eqn:pbiorthogonal})
and
 biconjugacy % (\ref{eqn:pbiconjugate})
are as follows:
\begin{eqnarray}
\LA\tvc{r}^\SH_i, \tvc{r}_j\RA
   &\PCONV&   \LA \vc{r}^\SH_i, \vc{r}^+_j\RA %\NON
\label{eqn:pbiorthogonal_Ll}
\\
   &=&   \LA \pR^\rmL_i(A^\T\Pinvt)\vc{r}^\SH_0,
         \SP \pR^\rmL_j(\Pinv A)\vc{r}^+_0\RA
, \NON
\\
\LA\tvc{p}^\SH_i, \tA\tvc{p}_j\RA
   &\PCONV&   \LA \vc{p}^\SH_i, (\Pinv A)\vc{p}^+_j\RA %\NON
\label{eqn:pbiconjugate_Ll}
\\
   &=&   \LA \pP^\rmL_i(A^\T\Pinvt)\vc{r}^\SH_0,
         \SP (\Pinv A) \pP^\rmL_j(\Pinv A)\vc{r}^+_0\RA
. \NON
\end{eqnarray}

In \Alg~\ref{alg:pbicg_corresp_Llpcgs},
the structures of 
$\vc{r}^+_k = \pR^\rmL_k(\Pinv A)\vc{r}^+_0$
and
$\vc{p}^+_k = \pP^\rmL_k(\Pinv A)\vc{r}^+_0$
are fixed,
and
their coefficient matrices are fixed as $\Pinv A$,
because
the initial residual vector is $\vc{r}^+_0$.
Therefore,
$\vc{x}_{k+1} = \vc{x}_k + \alpha^\rmL_k P^\rmL_k(\Pinv A)\vc{r}^+_0$
is structured,
where $\alpha^\rmL_k$ means (\ref{eqn:alpha_leftPBiCG});
and
\Alg~\ref{alg:pbicg_corresp_Llpcgs} is confirmed
to be the left-preconditioned system.
This ISRV is set as $\vc{r}^\SH_0=\vc{r}^+_0$.

For reference,
this algorithm can also be derived by the following conversion:
\begin{eqnarray}
&&
\tA       \PCONV \Pinv_L A \Pinv_R , \SP
\tvc{x}_k \PCONV \P_R    \vc{x}_k ,  \SP
\tvc{b}   \PCONV \Pinv_L \vc{b} ,    \SP %\NON
\label{eqn:pbicg_conv_Lb}
\\
&&
\tvc{r}_k \PCONV \P_R \vc{r}^+_k , \SP
\tvc{p}_k \PCONV \P_R \vc{p}^+_k , \SP %\NON
% \\
\tvc{r}^\SH_k \PCONV \Pinvt_R \vc{r}^\SH_k , \SP
\tvc{p}^\SH_k \PCONV \Pinvt_R \vc{p}^\SH_k.   \NON
\end{eqnarray}

If $\PO_L=\PO$ and $\PO_R=I$ are set,
then this is the same as (\ref{eqn:pbicg_corresp_Llpcgs}).

\subsection{Standard PBiCG}
\label{ssec:improved1_pbicg}

This is the most general algorithm for the PBiCG,
and it corresponds to the PCGS algorithm labeled Improved1 in \cite{itoh2019a}.
This algorithm is derived from the following preconditioning conversion
applied to \Alg~\ref{alg:pbicg_simple}:
\begin{eqnarray}
&&
\tA       \PCONV \Pinv_L A \Pinv_R, \SP
\tvc{x}_k \PCONV \P_R\vc{x}_k,  \SP
\tvc{b}   \PCONV \Pinv_L\vc{b}, \SP %\NON
\label{eqn:pbicg}
\\
&&
\tvc{r}_k \PCONV \Pinv_L\vc{r}_k, \SP
\tvc{p}_k \PCONV \P_R\vc{p}^+_k,  \SP %\NON
% \\
\tvc{r}^\SH_k \PCONV \Pinvt_R \vc{r}^\SH_k, \SP
\tvc{p}^\SH_k \PCONV \P^\T_L   \vc{p}^\flat_k. \NON
\end{eqnarray}
% \\

%%%%%%%%%%%%%%%%%%%%%%%%%%%%%%%%%%%%%%%%%%%%%%%%%%%%%%%%%%%%%
\refstepcounter{alg}
\label{alg:pbicg}
\noindent
{\bf Algorithm \thealg.
Standard PBiCG algorithm:}
\begin{indention}{\ALGWIDTHA}
%%%%%%%%%%%%%%%%%%%%%%%%%%%%%%%%%%%%%%%%%%%%%%%%%%%%%%%%%%%%%
\noindent
$\vc{x}_0$ is an initial guess,
$\SP \vc{r}_0= \vc{b}-A\vc{x}_0, \SP\SP$
set $\SP \beta_{-1}=0$,
\\
$\LA \TL{\vc{r}}^\SH_0, \TL{\vc{r}}_0 \RA
 \PCONV \LA \vc{r}^\SH_0, \Pinv \vc{r}_0 \RA
 \neq 0$,
e.g.,
$\vc{r}^\SH_0 = \Pinv \vc{r}_0, \SP$
 \\
For $k = 0, 1, 2, \cdots ,$ until convergence, Do:
\begin{eqnarray}
&& \vc{p}^+_k = \Pinv\vc{r}_k + \beta_{k-1}\vc{p}^+_{k-1},\SP\SP \NON \\
&& \vc{p}^\flat_k = \Pinvt\vc{r}^\SH_k + \beta_{k-1}\vc{p}^\flat_{k-1}, \NON \\
&& \alpha_k = \frac
            {\LA\vc{r}^\SH_k,   \P^{-1}\vc{r}_k\RA}
            {\LA\vc{p}^\flat_k, A\vc{p}^+_k\RA}, \NON \\
&& \vc{x}_{k+1} = \vc{x}_k + \alpha_k\vc{p}^+_k, \NON \\
&& \vc{r}_{k+1} = \vc{r}_k - \alpha_kA\vc{p}^+_k, \SP\SP \NON \\
&& \vc{r}^\SH_{k+1} = \vc{r}^\SH_k - \alpha_k A^\T\vc{p}^\flat_k, \NON \\
&& \beta_k = \frac
            {\LA\vc{r}^\SH_{k+1}, \Pinv\vc{r}_{k+1}\RA}
            {\LA\vc{r}^\SH_k,     \Pinv\vc{r}_k\RA}, \NON
\end{eqnarray}
End Do
\\

\end{indention}
%%%%%%%%%%%%%%%%%%%%%%%%%%%%%%%%%%%%%%%%%%%%%%%%%%%%%%%%%%%%%

In this algorithm,
the stopping criterion is (\ref{eqn:conv_judge}).

Although sometimes the ISRV is set such that
$( \vc{r}^\SH_0, \vc{r}_0 ) \neq 0,$ e.g., $\vc{r}^\SH_0 = \vc{r}_0 $, in many cases,
we will assume 
$ ( \tvc{r}^\SH_0, \tvc{r}_0 ) \neq 0$,
 e.g., $\vc{r}^\SH_0 = \Pinv\vc{r}_0$,
since
$  ( \tvc{r}^\SH_0, \tvc{r}_0 )
 =
   ( \Pinvt_R \vc{r}^\SH_0, \Pinv_L\vc{r}_0 )
 = ( \vc{r}^\SH_0, \Pinv \vc{r}_0 ) $
from (\ref{eqn:pbicg});
see section~\ref{sec:direction_preconditioning}.

The polynomials of the linear system are converted as
\begin{eqnarray}
\tvc{r}_k %\PCONV \Pinv_L\vc{r}_k
              &=& \pR^\rmL_k(\tA)\tvc{r}_0
% \\
               =  \Pinv_L\pR^\rmL_k(A\Pinv)\vc{r}_0, % \NON
\\
\tvc{p}_k %\PCONV \P_R\vc{p}^+_k
              &=& \pP^\rmL_k(\tA)\tvc{r}_0
% \\
               =  \Pinv_L \pP^\rmL_k(A\Pinv)\vc{r}_0, %\NON
\label{eqn:plinear_pol_p_conv2}
\end{eqnarray}
and
\begin{eqnarray}
\vc{r}_k &=&     \pR^\rmL_k(A\Pinv)\vc{r}_0 %, % \NON
          =  \PO \pR^\rmL_k(\Pinv A)\Pinv\vc{r}_0 ,
\label{eqn:plinear_pol2_r}
\\
\vc{p}^+_k &=& \Pinv \pP^\rmL_k(A\Pinv)\vc{r}_0 %.
            =  \pP^\rmL_k(\Pinv A)\Pinv\vc{r}_0 .
\label{eqn:plinear_pol2_p}
\end{eqnarray}

In the shadow system, we have
\begin{eqnarray}
\tvc{r}^\SH_k %\PCONV \Pinvt_R\vc{r}^\SH_k
              &=& \pR^\rmL_k(\tA^\T)\tvc{r}^\SH_0
% \\
               =  \Pinvt_R \pR^\rmL_k(A^\T \Pinvt)\vc{r}^\SH_0, \NON
\label{eqn:pshadow_pol_r_conv2}
\\
\tvc{p}^\SH_k %\PCONV \P^\T_L\vc{p}^\flat_k
              &=& \pP^\rmL_k(\tA^\T)\tvc{r}^\SH_0
% \\
               =  \Pinvt_R \pP^\rmL_k(A^\T \Pinvt)\vc{r}^\SH_0, \NON
\label{eqn:pshadow_pol_p_conv}
\end{eqnarray}
and
\begin{eqnarray}
\vc{r}^\SH_k &=& \pR^\rmL_k(A^\T\Pinvt)\vc{r}^\SH_0, \NON
\label{eqn:pshadow_pol2_r}
\\
\vc{p}^\flat_k &=& \Pinvt \pP^\rmL_k(A^\T\Pinvt)\vc{r}^\SH_0 . \NON
\label{eqn:pshadow_pol2_p}
\end{eqnarray}

The structures of
biorthogonality % (\ref{eqn:pbiorthogonal})
and
 biconjugacy % (\ref{eqn:pbiconjugate})
are as follows:
\begin{eqnarray}
\LA \tvc{r}^\SH_i, \tvc{r}_j\RA
   &\PCONV&   \LA \Pinvt_R\vc{r}^\SH_i, \Pinv_L\vc{r}_j\RA
    =    \LA \Pinvt\vc{r}^\SH_i, \vc{r}_j\RA
    =    \LA \vc{r}^\SH_i, \Pinv\vc{r}_j\RA %\NON
\label{eqn:pbiorthogonal_impr1}
\\
   &=&   \LA \pR^\rmL_i(A^\T\Pinvt)\vc{r}^\SH_0,\SP \Pinv \pR^\rmL_j(A\Pinv)\vc{r}_0\RA
, \NON
\\
\LA \tvc{p}^\SH_i, \tA\tvc{p}_j\RA
   &\PCONV&   \LA \P^\T_L\vc{p}^\flat_i, (\Pinv_LA\Pinv_R)(\P_R\vc{p}^+_j)\RA
    =    \LA \vc{p}^\flat_i, A\vc{p}^+_j\RA %\NON
\label{eqn:pbiconjugate_impr1}
\\
   &=&   \LA \Pinvt \pP^\rmL_i(A^\T\Pinvt)\vc{r}^\SH_0,\SP A \Pinv \pP^\rmL_j(A\Pinv)\vc{r}_0\RA. \NON
\end{eqnarray}

\begin{remark}
\label{rem:std_pbicg_prec_oper}
In \Alg~\ref{alg:pbicg},
the biorthogonal and biconjugate structures are not immediately apparent when either 
$\Pinv$ operates on the linear system
or
  $\Pinvt$ operates on the shadow system.
However,
\Alg~\ref{alg:pbicg} can be reduced to 
\Alg~\ref{alg:pbicg_corresp_Llpcgs} of the left system
by using $\vc{r}^+_k \equiv \Pinv\vc{r}_k$ and
$\vc{p}^\flat_k \mapsto \Pinvt\vc{p}^\SH_k$;
therefore,
\Alg~\ref{alg:pbicg} is coordinative to the left system.
The structure of the recurrence formula of the solution vector
 is
$ \vc{x}_{k+1} = \vc{x}_k + \alpha^\rmL_k P^\rmL_k(\Pinv A)\vc{r}^+_0
  \mapsto
  \vc{x}_{k+1} = \vc{x}_k + \alpha^\rmL_k \Pinv P^\rmL_k(A\Pinv)\vc{r}_0 $;
 this is obtained
by splitting $\vc{r}^+_0$. % in \Ssec~\ref{ssec:pbicg_Llprec}.
These structures
are verified
theoretically in section~\ref{sec:direction_preconditioning}
 and
numerically in section~\ref{sec:numerical_experiments}.
\end{remark}

\begin{remark}
\label{rem:std_pbicg_prec_pol}
We explicitly provided the equations for the right endpoints of
(\ref{eqn:plinear_pol2_r}) and (\ref{eqn:plinear_pol2_p}).
These are the final structures
for the setting of $\vc{r}^\SH_0 = \Pinv\vc{r}_0$
 (see Example 2 in the \App~\ref{appsec:bahavior_pol_pbicg}).
\end{remark}

\subsection{PBiCG corresponding to Improved2 }
\label{ssec:improved2_pbicg}

The PBiCG algorithm corresponding to the Improved2 PCGS algorithm in \cite{itoh2019a} (Improved2)
is derived from applying the following preconditioning conversion
to \Alg~\ref{alg:pbicg_simple}:
\begin{eqnarray}
&&
\tA        \PCONV \Pinv_L A \Pinv_R,  \SP
\tvc{x}_k  \PCONV \P_R \vc{x}_k, \SP
\tvc{b}    \PCONV \Pinv_L \vc{b}, \SP %\NON
\\
&&
\tvc{r}_k  \PCONV \Pinv_L \vc{r}_k, \SP
\tvc{p}_k  \PCONV \Pinv_L \vc{p}_k, \SP % \NON
% \\
\tvc{r}^\SH_k  \PCONV \Pinvt_R \vc{r}^\SH_k, \SP
\tvc{p}^\SH_k  \PCONV \Pinvt_R \vc{p}^\SH_k.  \NON
\label{eqn:improved2_pbicg}
\end{eqnarray}

This is different from the conversion applied to $\tvc{p}_k$ and $\tvc{p}^\SH_k$
in (\ref{eqn:pbicg}) for \Alg~\ref{alg:pbicg}.
\\

\newpage
%%%%%%%%%%%%%%%%%%%%%%%%%%%%%%%%%%%%%%%%%%%%%%%%%%%%%%%%%%%%%
\refstepcounter{alg}
\label{alg:improved2_pbicg}
\noindent
{\bf Algorithm \thealg.
PBiCG algorithm corresponding to Improved2:}
\begin{indention}{\ALGWIDTHB}
%%%%%%%%%%%%%%%%%%%%%%%%%%%%%%%%%%%%%%%%%%%%%%%%%%%%%%%%%%%%%
\noindent
$\vc{x}_0$ is an initial guess,
$\SP\vc{r}_0= \vc{b}-A\vc{x}_0, \SP\SP$
set $\SP\beta_{-1}=0$,
\\
$\LA \tvc{r}^\SH_0, \tvc{r}_0 \RA
 \PCONV \LA\vc{r}^\SH_0, \Pinv \vc{r}_0\RA
 \neq 0$,
e.g.,
$\vc{r}^\SH_0 = \Pinv\vc{r}_0, \SP$
 \\
For $k = 0, 1, 2, \cdots ,$ until convergence, Do:
\begin{eqnarray}
&& \vc{p}_k = \vc{r}_k + \beta_{k-1}\vc{p}_{k-1}, \SP %\NON
\label{eqn:improved2_pbicg_pol_p}
\\
&& \LA\Pinv\vc{p}_k = \Pinv\vc{r}_k + \beta_{k-1}\Pinv\vc{p}_{k-1},\SP\RA %\NON
\label{eqn:improved2_pbicg_pol_p2}
\\
&& \vc{p}^\SH_k = \vc{r}^\SH_k + \beta_{k-1}\vc{p}^\SH_{k-1}, \NON
\\
&& \alpha_k = \frac
            {\LA\Pinvt\vc{r}^\SH_k, \vc{r}_k\RA}
            {\LA\Pinvt\vc{p}^\SH_k, A\Pinv\vc{p}_k\RA}
 = \frac
            {\LA\vc{r}^\SH_k, \Pinv\vc{r}_k\RA}
            {\LA\Pinvt\vc{p}^\SH_k, A\Pinv\vc{p}_k\RA} , \NON
\\
&& \vc{x}_{k+1} = \vc{x}_k + \alpha_k\Pinv\vc{p}_k, \NON \\
&& \vc{r}_{k+1} = \vc{r}_k - \alpha_kA\Pinv\vc{p}_k, \SP\SP \NON \\
&& \vc{r}^\SH_{k+1} = \vc{r}^\SH_k - \alpha_kA^\T\Pinvt\vc{p}^\SH_k, \NON \\
&& \beta_k = \frac
            {\LA\Pinvt\vc{r}^\SH_{k+1}, \vc{r}_{k+1}\RA}
            {\LA\Pinvt\vc{r}^\SH_k, \vc{r}_k\RA}
 = \frac
            {\LA\vc{r}^\SH_{k+1}, \Pinv\vc{r}_{k+1}\RA}
            {\LA\vc{r}^\SH_k, \Pinv\vc{r}_k\RA} , \NON
\end{eqnarray}
End Do
\\

\end{indention}
%%%%%%%%%%%%%%%%%%%%%%%%%%%%%%%%%%%%%%%%%%%%%%%%%%%%%%%%%%%%%

In this algorithm%
\footnote{
Practically,
(\ref{eqn:improved2_pbicg_pol_p2}) is implemented as
$\vc{p}^+_k\equiv\Pinv\vc{p}_k$,
therefore,
(\ref{eqn:improved2_pbicg_pol_p}) is needless,
and
its preconditioning operations in the iterated part are
just $\Pinvt\vc{p}^\SH_k$ and $\Pinv\vc{r}_k$.
}
,
the stopping criterion is (\ref{eqn:conv_judge}).

The structure of the biorthogonality is the same
as that of (\ref{eqn:pbiorthogonal_impr1}) for \Alg~\ref{alg:pbicg},
because the same preconditioning conversion is used for
$\tvc{r}_k$ and $\tvc{r}^\SH_k$.
The probing direction polynomials of the linear and shadow systems
are converted as
\begin{eqnarray}
\tvc{p}_k %\PCONV \Pinv_L\vc{p}_k
              &=& \pP^\rmL_k(\tA)\tvc{r}_0
               =  \Pinv_L \pP^\rmL_k(A\Pinv)\vc{r}_0 , % \NON
\reeqno{\ref{eqn:plinear_pol_p_conv2}'}
\\
\tvc{p}^\SH_k %\PCONV \Pinvt_R\vc{p}^\SH_k
              &=& \pP^\rmL_k(\tA^\T)\tvc{r}^\SH_0
               =  \Pinvt_R \pP^\rmL_k(A^\T \Pinvt)\vc{r}^\SH_0, %\NON
\label{eqn:pshadow_pol2_p_conv}
\end{eqnarray}
and
\begin{eqnarray}
\vc{p}_k &=&    \pP^\rmL_k(A\Pinv)\vc{r}_0 %, \NON
          =  \PO\pP^\rmL_k(\Pinv A)\Pinv\vc{r}_0 ,
\label{eqn:plinear_pol3_p}
\\
\vc{p}^\SH_k &=& \pP^\rmL_k(A^\T \Pinvt)\vc{r}^\SH_0 . %\NON
\label{eqn:pshadow_pol3_p}
\end{eqnarray}

The structure of the biconjugacy is
\begin{eqnarray}
\LA \tvc{p}^\SH_i, \tA\tvc{p}_j \RA
   &\PCONV&   \LA \Pinvt_R \vc{p}^\SH_i,\SP (\Pinv_L A\Pinv_R)(\Pinv_L \vc{p}_j)\RA
    =  \LA \Pinvt \vc{p}^\SH_i,\SP A\Pinv \vc{p}_j\RA
\NON
\\
   &=& \LA \Pinvt \pP^\rmL_i(A^\T\Pinvt)\vc{r}^\SH_0,\SP A \Pinv \pP^\rmL_j(A\Pinv)\vc{r}_0\RA .
\NON
% \\
\end{eqnarray}
This structure is the same as that of (\ref{eqn:pbiconjugate_impr1})
for \Alg~\ref{alg:pbicg}.
This ISRV is set as $\vc{r}^\SH_0=\Pinv\vc{r}_0$.

\begin{remark}
\label{rem:impr2_pbicg_prec_oper}
As before, in \Alg~\ref{alg:improved2_pbicg},
the biorthogonal and biconjugate structures are not immediately apparent when either
$\Pinv$  operates on the linear system
or
  $\Pinvt$ operates on the shadow system.
However,
the structure of the recurrence formula of the solution vector
 is again
$ \vc{x}_{k+1} = \vc{x}_k + \alpha^\rmL_k \Pinv P^\rmL_k(A\Pinv)\vc{r}_0$,
because
\Alg~\ref{alg:improved2_pbicg} is equivalent to \Alg~\ref{alg:pbicg}
on the $\alpha_k$ and $\beta_k$,
the residual and shadow residual vectors, respectively.
These properties are verified
theoretically in section~\ref{sec:direction_preconditioning}
 and
numerically in section~\ref{sec:numerical_experiments}.
\end{remark}

\begin{remark}
\label{rem:impr2_pbicg_prec_pol}
We explicitly provided (\ref{eqn:plinear_pol3_p}) for the right endpoint.
This is the final structure obtained for $\vc{r}^\SH_0 = \Pinv\vc{r}_0$
 (see \Rem~{\it \ref{rem:std_pbicg_prec_pol}}).
\end{remark}

\subsection{Characteristic features of the four PBiCG algorithms}
\label{ssec:analysis_of_4pcgs}

In this section, we present the characteristics of each of the PBiCG algorithms.
These include
the construction of the ISRV,
the biorthogonal and biconjugate structures of the
$\alpha_k$ and $\beta_k$,
and
the structures of the recurrence formula of the solution vector.
In the following equations, the underlined inner products are
the typical descriptions on $\alpha_k$ and $\beta_k$.

\begin{itemize}

\item
PBiCG corresponding to the conventional PCGS (\Alg~\ref{alg:pbicg_corresp_conv_pcgs}):
\begin{eqnarray}
\vc{r}^\flat_0 &=& \vc{r}_0, \NON
\\
\LA\tvc{r}^\SH_\i, \tvc{r}_\j\RA
% \\
   &=&   \LA \pR^\rmR_\i(\Pinvt A^\T)\vc{r}^\flat_0, \SP \pR^\rmR_\j(A \Pinv)\vc{r}_0\RA
    =    \UL{ \LA \vc{r}^\flat_\i, \vc{r}_\j\RA }
, \NON
\\
\LA\tvc{p}^\SH_\i, \tA\tvc{p}_\j\RA
% \\
   &=&   \LA \pP^\rmR_\i(\Pinvt A^\T)\vc{r}^\flat_0, \SP \L(A\Pinv\R) \pP^\rmR_\j(A\Pinv)\vc{r}_0\RA
\NON
\\
    &=&    \UL{ \LA \vc{p}^\flat_\i, (A \Pinv)\vc{p}_\j\RA }
, \NON
\\
\P\vc{x}_{k+1} &=& \P\vc{x}_k + \alpha^\rmR_k P^\rmR_k(A\Pinv)\vc{r}_0 . \NON
\end{eqnarray}

\item
Left-PBiCG (\Alg~\ref{alg:pbicg_corresp_Llpcgs}):
\begin{eqnarray}
\vc{r}^\sharp_0 &=& \vc{r}^+_0, \NON
\\
\LA\tvc{r}^\SH_\i, \tvc{r}_\j\RA
   &=&   \LA \pR^\rmL_\i(A^\T\Pinvt)\vc{r}^\SH_0, \SP \pR^\rmL_\j(\Pinv A)\vc{r}^+_0\RA
    =    \UL{ \LA \vc{r}^\SH_\i, \vc{r}^+_\j\RA }
, \NON
\\
\LA\tvc{p}^\SH_\i, \tA\tvc{p}_\j\RA
   &=&   \LA \pP^\rmL_\i(A^\T\Pinvt)\vc{r}^\SH_0, \SP \L(\Pinv A\R)\pP^\rmL_\j(\Pinv A)\vc{r}^+_0\RA
\NON
\\ 
    &=&    \UL{ \LA \vc{p}^\SH_\i, (\Pinv A)\vc{p}^+_\j\RA }
, \NON
\\
\vc{x}_{k+1} &=& \vc{x}_k + \alpha^\rmL_k P^\rmL_k(\Pinv A)\vc{r}^+_0
. \NON
\end{eqnarray}

\item
Standard PBiCG (\Alg~\ref{alg:pbicg}):
\begin{eqnarray}
\vc{r}^\SH_0 &=& \Pinv \vc{r}_0, \NON
\\
\LA\tvc{r}^\SH_\i, \tvc{r}_\j\RA
   &=&   \LA \pR^\rmL_\i(A^\T\Pinvt)\vc{r}^\SH_0, \SP  \Pinv \pR^\rmL_\j(A\Pinv)\vc{r}_0\RA
    =    \UL{ \LA \vc{r}^\SH_\i, \Pinv\vc{r}_\j\RA }
, \NON
\\
\LA\tvc{p}^\SH_\i, \tA\tvc{p}_\j\RA
   &=&   \LA \Pinvt \pP^\rmL_\i(A^\T\Pinvt)\vc{r}^\SH_0, \SP A\Pinv \pP^\rmL_\j(A\Pinv)\vc{r}_0\RA
\NON
\\
    &=&    \UL{ \LA \vc{p}^\flat_\i, A\vc{p}^+_\j\RA }
, \NON
\\
\vc{x}_{k+1} &=& \vc{x}_k + \alpha^\rmL_k \Pinv P^\rmL_k(A\Pinv)\vc{r}_0
. \NON
\end{eqnarray}

\item
PBiCG corresponding to Improved2 (\Alg~\ref{alg:improved2_pbicg}):
\begin{eqnarray}
\vc{r}^\SH_0 &=& \Pinv \vc{r}_0, \NON
\\
\LA\tvc{r}^\SH_\i, \tvc{r}_\j\RA
   &=&   \LA \Pinvt \pR^\rmL_\i(A^\T\Pinvt)\vc{r}^\SH_0, \SP \pR^\rmL_\j(A\Pinv)\vc{r}_0\RA \NON
\\
   &=&   \UL{ \LA \Pinvt\vc{r}^\SH_\i, \vc{r}_\j\RA }
    =         \LA \vc{r}^\SH_\i, \Pinv \vc{r}_\j\RA
, \NON
\\
\LA\tvc{p}^\SH_\i, \tA\tvc{p}_\j\RA
   &=&   \LA \Pinvt \pP^\rmL_\i(A^\T\Pinvt)\vc{r}^\SH_0, \SP A \Pinv \pP^\rmL_\j(A\Pinv)\vc{r}_0\RA \NON
\\
   &=&   \UL{ \LA \Pinvt\vc{p}^\SH_\i, A\Pinv \vc{p}_\j\RA }
    =         \LA \vc{p}^\SH_\i, (\Pinv A)(\Pinv \vc{p}_\j) \RA
, \NON
\\
\vc{x}_{k+1} &=& \vc{x}_k + \alpha^\rmL_k \Pinv P^\rmL_k(A\Pinv)\vc{r}_0
. \NON
\end{eqnarray}

\end{itemize}

Although, superficially, it appears that the
solution vector has the same recurrence relation in both
\Alg~\ref{alg:pbicg_corresp_conv_pcgs}
and 
\Alg~\ref{alg:improved2_pbicg}
($\vc{x}_{k+1} = \vc{x}_k + \alpha_k\Pinv\vc{p}_k$),
they belong to different systems
because
in \Alg~\ref{alg:pbicg_corresp_conv_pcgs}, we have
$\alpha^\rmR_k$ and
$\vc{p}_k = \pP^\rmR_k(A\Pinv)\vc{r}_0 \equiv \vc{p}^\rmR_k$,
whereas in
\Alg~\ref{alg:improved2_pbicg}, we have
$\alpha^\rmL_k$ and
$\vc{p}_k = \pP^\rmL_k(A\Pinv)\vc{r}_0 \equiv \vc{p}^\rmL_k$.

We have the following proposition about
the direction of the preconditioning conversion%
\footnote{
Although this property has been repeatedly discussed in the literature,
it should be considered when evaluating
the direction of a preconditioned system.
}%
.
\begin{proposition}[Congruency]
\label{prop:congruency}
There is congruence to a PBiCG algorithm
in the direction of the preconditioning conversion.
\end{proposition}
\begin{proof}
We have already shown the following instances:
the
PBiCG of the right system
 (\Alg~\ref{alg:pbicg_corresp_conv_pcgs}) can be derived
from the two-sided conversion (\ref{eqn:pbicg_corresp_conv_pcgs});
if $\PO_L=I$ and $\PO_R=\PO $,
the conversion of (\ref{eqn:pbicg_corresp_conv_pcgs}) is reduced
to that of (\ref{eqn:pbicg_right_system_pcgs}),
then \Alg~\ref{alg:pbicg_corresp_conv_pcgs} is derived.
Still if $\PO_L=\PO$, $\PO_R=I $,
then \Alg~\ref{alg:pbicg_corresp_conv_pcgs} can be derived.
Each of the other preconditioned algorithms
(\Alg~\ref{alg:pbicg_corresp_Llpcgs}, \ref{alg:pbicg},
and \ref{alg:improved2_pbicg}) has the same relationship to 
its corresponding preconditioning conversion.
\end{proof}

%%%%%%%%%%%%%%%%%%%%%%%%%%%%%%%%%%%%%%%%%%%%%%%%%%%%%%%%%%%%%%%%%%%%%%%%%%

\begin{proposition}
\label{prop:pbicg_biortho}
On the structure of biorthogonality $(\tvc{r}^\SH_k, \tvc{r}_k)$
in the iterated part of each PBiCG in the appeared four algorithms,
there exists a single preconditioning operator
between
  $\vc{r}_k$ (basic form of the residual vector)
and
  $\vc{r}^\SH_k$ (basic form of the shadow residual vector),
such that
  $\Pinv$ operates on $\vc{r}_k$
or
  $\Pinvt$ operates on $\vc{r}^\SH_k$.

Here,
  the basic form of the residual vector of a linear system includes
  its polynomial structure of
  $\vc{r}_k=\pR_k(A\Pinv)\vc{r}_0$,
and
  the basic form of the shadow residual vector includes
  its polynomial structure of
  $\vc{r}^\SH_k=\pR_k(A^\T\Pinvt)\vc{r}^\SH_0$;
these vectors and polynomials are not considered when
setting the ISRV.
\end{proposition}

{\bf Proof }
1) From the viewpoint of the matrix and vector structure
of each algorithm:

We split
$\vc{r}^\flat_0 \mapsto \Pinvt\vc{r}^\SH_0$ and
$\vc{r}^+_k     \mapsto \Pinv \vc{r}_k$,
in Algorithms \ref{alg:pbicg_corresp_conv_pcgs} to \ref{alg:improved2_pbicg}
then we set
\begin{eqnarray}
\L( \tvc{r}^\SH_k, \tvc{r}_k \R)
   &=&   \UL{ \L( \vc{r}^\flat_k, \SP \vc{r}_k \R) }
 \mapsto \L( \Pinvt \vc{r}^\SH_k, \SP \vc{r}_k \R) , \NON
\label{eqn:bicg_precb_biortho_detail}
\\
\L( \tvc{r}^\SH_k, \tvc{r}_k \R)
   &=&   \UL{ \L(\vc{r}^\SH_k, \SP \vc{r}^+_k\R) }
 \mapsto \L( \vc{r}^\SH_k, \SP \Pinv \vc{r}_k \R) , \NON
\label{eqn:bicg_Llprec_biortho_detail}
\\
\L( \tvc{r}^\SH_k, \tvc{r}_k \R)
   &=&   \UL{ \L(\vc{r}^\SH_k, \SP \Pinv\vc{r}_k\R) }, \NON
\label{eqn:bicg_impr1_biortho_detail}
\\
\L( \tvc{r}^\SH_k, \tvc{r}_k \R)
   &=&   \UL{ \L(\Pinvt\vc{r}^\SH_k, \SP \vc{r}_k\R) }
    =    \L( \vc{r}^\SH_k, \SP \Pinv\vc{r}_k \R). \NON
\label{eqn:bicg_impr2_biortho_detail}
\end{eqnarray}

The underlined inner products are the typical descriptions for the various PBiCG.

In addition,
for the two-sided conversion,
we obtain
\begin{eqnarray}
\L( \tvc{r}^\SH_k, \tvc{r}_k \R)
   &=&   \L(\Pinvt_R\vc{r}^\SH_k, \SP \Pinv_L\vc{r}_k\R)
    =    \L(\Pinvt\vc{r}^\SH_k, \SP \vc{r}_k\R)
    =    \L( \vc{r}^\SH_k, \SP \Pinv\vc{r}_k \R) . \NON
\label{eqn:bicg_two_biortho_detail}
\qquad
\hfill\Box
\end{eqnarray}

{\bf Proof }
2) From the viewpoint of the polynomial of the residual vector:

We split
$\vc{r}^\flat_0 \mapsto \Pinvt\vc{r}^\SH_0$ and
$\vc{r}^+_k     \mapsto \Pinv \vc{r}_k$ in Algorithms \ref{alg:pbicg_corresp_conv_pcgs} to \ref{alg:improved2_pbicg},
then we set
\begin{eqnarray}
\LA\TL{\vc{r}}^\SH_\i, \TL{\vc{r}}_\j\RA
   &=&   \UL{ \LA \pR_\i(\Pinvt A^\T)\vc{r}^\flat_0, \SP \pR_\j(A \Pinv)\vc{r}_0\RA }
\NON
\label{eqn:bicg_precb_biortho_pol}
\\
   &\mapsto& \LA \Pinvt \pR_\i(A^\T\Pinvt )\vc{r}^\SH_0, \SP \pR_\j(A \Pinv)\vc{r}_0\RA ,
\NON
\\
\LA\TL{\vc{r}}^\SH_\i, \TL{\vc{r}}_\j\RA
   &=&   \UL{ \LA \pR_\i(A^\T\Pinvt)\vc{r}^\SH_0, \SP \pR_\j(\Pinv A)\vc{r}^+_0\RA }
\NON
\label{eqn:bicg_Llprec_biortho_pol}
\\
   &\mapsto& \LA \pR_\i(A^\T\Pinvt)\vc{r}^\SH_0, \SP \Pinv \pR_\j(A\Pinv )\vc{r}_0\RA ,
\NON
\\
\LA\TL{\vc{r}}^\SH_\i, \TL{\vc{r}}_\j\RA
   &=&   \UL{ \LA \pR_\i(A^\T\Pinvt)\vc{r}^\SH_0, \SP  \Pinv \pR_\j(A\Pinv)\vc{r}_0\RA } ,
\NON
\label{eqn:bicg_impr1_biortho_pol}
\\
\LA\TL{\vc{r}}^\SH_\i, \TL{\vc{r}}_\j\RA
   &=&   \UL{ \LA \Pinvt \pR_\i(A^\T\Pinvt)\vc{r}^\SH_0, \SP \pR_\j(A\Pinv)\vc{r}_0\RA } . % \NON
\NON
\label{eqn:bicg_impr2_biortho_pol}
\end{eqnarray}

The underlined inner products are the structures of the polynomial
corresponding to the residual vectors in each PBiCG.

In addition,
for the two-sided conversion,
we obtain
\begin{eqnarray}
\LA\TL{\vc{r}}^\SH_\i, \TL{\vc{r}}_\j\RA
   &=&   \LA \pR_\i(\Pinvt_R A^\T\Pinvt_L)\Pinvt_R\vc{r}^\SH_0,
      \SP  \pR_\j(\Pinv_L A\Pinv_R)\Pinv_L\vc{r}_0\RA  \NON
\\
   &=&   \LA \Pinvt_R \pR_\i(A^\T\Pinvt)\vc{r}^\SH_0,
      \SP  \Pinv_L \pR_\j(A\Pinv)\vc{r}_0\RA  \NON
\\
   &=&   \LA \Pinvt\pR_\i(A^\T\Pinvt)\vc{r}^\SH_0, \SP  \pR_\j(A\Pinv)\vc{r}_0\RA
\NON
\\
   &=&   \LA \pR_\i(A^\T\Pinvt)\vc{r}^\SH_0, \SP  \Pinv \pR_\j(A\Pinv)\vc{r}_0\RA .
\NON
\label{eqn:bicg_two_biortho_pol}
\qquad
\hfill\Box
\end{eqnarray}

\begin{corollary}
\label{cor:pbicg_biconj}
In the biconjugate structure $(\tvc{p}^\SH_k, \tA\tvc{p}_k)$
in the iterated part of each PBiCG algorithm,
there exists a single preconditioning operator
between
  $A$ (coefficient matrix)
and
  $\vc{p}^\SH_k$ (basic form of the shadow probing direction vector),
such that $\Pinv$ operates on $A$
or
  $\Pinvt$ operates on $\vc{p}^\SH_k$;
furthermore,
there exists a single preconditioning operator
between
  $A$
and
  $\vc{p}_k$ (basic form of the probing direction vector).

Here,
  the basic form of the probing direction vector of a linear system includes
  the polynomial structure of
  $\vc{p}_k=\pP_k(A\Pinv)\vc{p}_0$,
and
  the basic form of the shadow probing direction vector includes
  the polynomial structure of
  $\vc{p}^\SH_k=\pP_k(A^\T\Pinvt)\vc{p}^\SH_0$.
These vectors and polynomials are not considered when
setting the ISRV.
\end{corollary}

%%%%%%  section  %%%%%%%%%%%%%%%%%%%%%%%%%%%%%%%%%%%%%%%%%%%%%%%%%%%%%%%%%%%%%%%%%
\section{Switching the direction of the preconditioned system for the BiCG method}
\label{sec:direction_preconditioning}

From the analyses presented in the previous sections and in \cite{itoh2019a},
we know that the intrinsic biorthogonal and biconjugate structures
of the preconditioned system
are the same for each of the four PBiCG algorithms
and their corresponding PCGS algorithms,
and this is independent of the setting of the ISRV.
We now consider the other
factor that can switch the direction of the preconditioning:
the construction and setting of the ISRV.

As stated above, if the coefficient matrix $A$ is SPD
and $\vc{r}^\SH_0=\vc{r}_0$,
then the BiCG method is mathematically equivalent to
the CG method.
However,
the BiCG method solves systems of linear equations
that correspond to a nonsymmetric coefficient matrix,
and
the ISRV $\vc{r}^\SH_0$ is usually
regarded as arbitrary,
providing that $(\vc{r}^\SH_0, \vc{r}_0)\neq 0$.
On the other hand,
we may construct an arbitrary vector $\vc{r}^\SH_0 = U\vc{r}_0$,
such that $(\vc{r}^\SH_0, U \vc{r}_0)\neq 0$.
Here,
the matrix $U$ is unprescribed.
Obviously,
$U\vc{r}_0$ can generate random vectors.
However,
by the appropriate construction of $U$,
the BiCG can be reduced to the other method \cite{abe2010, abe2007}.
We show this result as the following proposition.
\begin{proposition}
\label{prop:bicr_by_isrv}
If we let $U=A^\T$ when $\vc{r}^\SH_0=U\vc{r}_0$ in the BiCG method, then we obtain
the biconjugate residual (BiCR) method \cite{sogabe2009, sogabe2004} %.
\footnote{
A series of product-type methods based on the BiCR have been proposed 
by Sogabe et al. \cite{sogabe2004};
these methods are based on an idea presented in \cite{zhang1997}.
The BiCR method was described in \cite{sogabe2009},
in a discussion of the product-type methods based on it.
Other product-type methods based
on the BiCR have been proposed \cite{abe2010,abe2007};
their derivation is different from that in \cite{sogabe2004},
and
these methods can be implemented more easily than that of \cite{sogabe2004}.
Note that the latter method can only be implemented to multiply $\vc{r}_0$ by $A^\T$ as the ISRV, 
that is, $U=A^\T$.
However,
{\Ref}s \cite{abe2010,abe2007} describe setting the ISRV to $A^\T\vc{r}^\SH_0$,
if $U=A^{-\T}$ at $\vc{r}^\SH_0 = U\vc{r}_0$;
these BiCR-type methods are then reduced to BiCG-type methods
(also see \Rem~{\it \ref{rem:arbitrary_isrv}}).
}.
\end{proposition}

\begin{theorem}
\label{thm:dir_prec_by_isrv}
The direction of a preconditioned system for the BiCG method is switched
by the construction and setting of the ISRV. %,
\end{theorem}

{\bf Proof }
It is sufficient to prove the following cases regarding the biorthogonality.
The biconjugacy can be proven in a similar manner.

Case I: 
If $\L(\tvc{r}^\SH_0, \tvc{r}_0 \R) \neq 0$, then $\tvc{r}^\SH_0=\tvc{r}_0$.

We mention the following special case for future reference.

Case II:
If $\L(\tvc{r}^\SH_0, \tU \tvc{r}_0 \R) \neq 0$, then $\tvc{r}^\SH_0=\tU \tvc{r}_0$,
$\SP\SP$
($\tU$ : preconditioned system of $U$).

\noindent
I.~ The case of $\tU=I$, such that 
$\L(\tvc{r}^\SH_0, \tU\tvc{r}_0 \R) = \L(\tvc{r}^\SH_0, \tvc{r}_0 \R) \neq 0$ :

With the equation $\tvc{r}^\SH_0 = \tvc{r}_0$,
we may construct the following three items.
Each item has two verifications,
the first one directly applies the ISRV to the polynomials
of the preconditioned system,
the second applies the ISRV to the polynomials
of the standard PBiCG, which have the same form in all items.
The double-underlined equations show the construction of the ISRV that is
specialized for switching the direction of a preconditioned system;
when right-hand side is given (we term ``setting''),
the direction is fixed.

\begin{enumerate}
\renewcommand{\labelenumi}{\arabic{enumi})}
\item The left-preconditioned system
(ISRV1: \UL{\UL{$\vc{r}^\SH_0=\Pinv\vc{r}_0$}}):

If $\tvc{r}^\SH_0 = \Pinvt_R\vc{r}^\SH_0, \SP \tvc{r}_0 = \Pinv_L\vc{r}_0$,

 $\L( \tvc{r}^\SH_0, \tvc{r}_0 \R)
= \L( \Pinvt_R\vc{r}^\SH_0, \Pinv_L\vc{r}_0 \R) 
= \L( \vc{r}^\SH_0, \Pinv\vc{r}_0 \R) \neq 0$,
then
  $\vc{r}^\SH_0 = \Pinv\vc{r}_0 $.

This is equivalent to
  $\tvc{r}^\SH_0 = \vc{r}^\SH_0, \SP \tvc{r}_0 = \Pinv\vc{r}_0$.

$\L( \tvc{r}^\SH_k, \tvc{r}_k \R)
= \L( \pR_k(\tA^\T)\tvc{r}^\SH_0, \pR_k(\tA)\tvc{r}_0 \R)
= \L( \pR_k(A^\T\Pinvt)\vc{r}^\SH_0, \pR_k(\Pinv A)(\Pinv\vc{r}_0) \R)$

$= \L( \pR^\rmL_k(A^\T\Pinvt)(\Pinv\vc{r}_0),
 \pR^\rmL_k(\Pinv A)(\Pinv\vc{r}_0) \R)$.

In the standard PBiCG with \UL{\UL{$\vc{r}^\SH_0=\Pinv\vc{r}_0$}}
that constructs the left system,

$\L( \tvc{r}^\SH_k, \tvc{r}_k \R)
= \L( \pR_k(A^\T\Pinvt)\SP\UL{\UL{\vc{r}^\SH_0}},\SP \Pinv\pR_k(A\Pinv)\vc{r}_0 \R)$

$= \L( \pR^\rmL_k(A^\T\Pinvt)(\UL{\UL{\Pinv\vc{r}_0}}),
 \pR^\rmL_k(\Pinv A)(\Pinv\vc{r}_0) \R)$.

\item The right-preconditioned system
(ISRV2: \UL{\UL{$\vc{r}^\SH_0=\PO^\T\vc{r}_0$}}):

If $\tvc{r}^\SH_0 = \Pinvt_R\vc{r}^\SH_0,\SP \tvc{r}_0 = \Pinv_L\vc{r}_0$,

 $\L( \tvc{r}^\SH_0, \tvc{r}_0 \R)
= \L( \Pinvt_R\vc{r}^\SH_0, \Pinv_L\vc{r}_0 \R)
= \L( \Pinvt\vc{r}^\SH_0, \vc{r}_0 \R) \neq 0$,
then
  $\Pinvt\vc{r}^\SH_0 \equiv \vc{r}^\flat_0 = \vc{r}_0$
or
  $\vc{r}^\SH_0 = \PO^\T\vc{r}_0$.

This is equivalent to 
  $\tvc{r}^\SH_0 = \Pinvt\vc{r}^\SH_0 \equiv \vc{r}^\flat_0,\SP \tvc{r}_0 = \vc{r}_0$
.

$\L( \tvc{r}^\SH_k, \tvc{r}_k \R)
= \L( \pR_k(\tA^\T)\tvc{r}^\SH_0, \pR_k(\tA)\tvc{r}_0 \R)
= \L( \pR_k(\Pinvt A^\T)\Pinvt\vc{r}^\SH_0, \pR_k(A\Pinv)\vc{r}_0 \R)$

$ \equiv
  \L( \pR_k(\Pinvt A^\T)\vc{r}^\flat_0, \pR_k(A\Pinv)\vc{r}_0 \R)$
$= \L( \pR^\rmR_k(\Pinvt A^\T)\vc{r}_0, \pR^\rmR_k(A\Pinv)\vc{r}_0 \R)$.

In the standard PBiCG with \UL{\UL{$\vc{r}^\SH_0=\PO^\T\vc{r}_0$}}
that constructs the right system,

$\L( \tvc{r}^\SH_k, \tvc{r}_k \R)
= \L( \pR_k(A^\T\Pinvt)\SP\UL{\UL{\vc{r}^\SH_0}},\SP
           \Pinv\pR_k(A\Pinv)\vc{r}_0 \R)$

$= \L( \Pinvt\pR^\rmR_k(A^\T\Pinvt)(\UL{\UL{\PO^\T\vc{r}_0}}),
           \pR^\rmR_k(A\Pinv)\vc{r}_0 \R)$
           
$= \L( \pR^\rmR_k(\Pinvt A^\T)\vc{r}_0,
           \pR^\rmR_k(A\Pinv)\vc{r}_0 \R)$.

\item The two-sided preconditioned system
 (ISRV3: \UL{\UL{$\vc{r}^\SH_0=\PO^\T_R\Pinv_L\vc{r}_0$}}):

If $\tvc{r}^\SH_0 = \Pinvt_R\vc{r}^\SH_0,\SP \tvc{r}_0 = \Pinv_L\vc{r}_0$,

 $\L( \tvc{r}^\SH_0, \tvc{r}_0 \R)
= \L( \Pinvt_R\vc{r}^\SH_0, \Pinv_L\vc{r}_0 \R) \neq 0$,
then
  $\Pinvt_R\vc{r}^\SH_0 = \Pinv_L\vc{r}_0$
or
  $\vc{r}^\SH_0 = \PO^\T_R\Pinv_L\vc{r}_0$.

This is obviously equivalent to 
  $\tvc{r}^\SH_0 = \Pinvt_R\vc{r}^\SH_0,\SP \tvc{r}_0 = \Pinv_L\vc{r}_0$.

$\L( \tvc{r}^\SH_k, \tvc{r}_k \R)
= \L( \pR_k(\tA^\T)\tvc{r}^\SH_0, \pR_k(\tA)\tvc{r}_0 \R)$

$= \L( \pR_k(\Pinvt_R A^\T\Pinvt_L )(\Pinvt_R\vc{r}^\SH_0),
                     \pR_k(\Pinv_L A\Pinv_R)(\Pinv_L\vc{r}_0) \R)$

$= \L( \pR^\rmW_k(\Pinvt_R A^\T\Pinvt_L )(\Pinv_L\vc{r}_0),
                     \pR^\rmW_k(\Pinv_L A\Pinv_R)(\Pinv_L\vc{r}_0) \R)$.

In the standard PBiCG with \UL{\UL{$\vc{r}^\SH_0 = \PO^\T_R\Pinv_L\vc{r}_0$}}
that constructs the two-sided system,

$\L( \tvc{r}^\SH_k, \tvc{r}_k \R)
= \L( \pR_k(A^\T\Pinvt)\SP\UL{\UL{\vc{r}^\SH_0}},\SP
                     \Pinv\pR_k(A\Pinv)\vc{r}_0 \R)$

$= \L( \pR^\rmW_k(A^\T\Pinvt)(\UL{\UL{\PO^\T_R\Pinv_L\vc{r}_0}}),\SP
                     \Pinv_R\Pinv_L\pR^\rmW_k(A\Pinv)\vc{r}_0 \R)$

$= \L( \PO^\T_R\pR^\rmW_k(\Pinvt_R A^\T\Pinvt_L)(\Pinv_L\vc{r}_0),\SP
                     \Pinv_R\pR^\rmW_k(\Pinv_LA\Pinv_R)(\Pinv_L\vc{r}_0) \R)$

$= \L( \pR^\rmW_k(\Pinvt_R A^\T\Pinvt_L)(\Pinv_L\vc{r}_0),\SP
                     \pR^\rmW_k(\Pinv_LA\Pinv_R)(\Pinv_L\vc{r}_0) \R)$.
\end{enumerate}

\noindent
II.~The case of an arbitrary $\tU$,
 such that $\L(\tvc{r}^\SH_0, \tU \tvc{r}_0 \R) \neq 0$ :

\begin{itemize}
\item[] If $\tvc{r}^\SH_0 = \Pinvt_R\vc{r}^\SH_0, \tvc{r}_0 = \Pinv_L\vc{r}_0$,

 $\L( \tvc{r}^\SH_0, \SP \tU \tvc{r}_0 \R)
= \L( \Pinvt_R\vc{r}^\SH_0, \SP \tU \Pinv_L\vc{r}_0 \R) \neq 0$,

then
  $\Pinvt_R\vc{r}^\SH_0 = \tU \Pinv_L\vc{r}_0$
or
  $\vc{r}^\SH_0 = \PO^\T_R \tU \Pinv_L\vc{r}_0$.
  
Here,
if $\tU = I$,
   then the two-sided system is constructed,
   because $\UL{\UL{\vc{r}^\SH_0 =}}$ $\UL{\UL{\PO^\T_R\Pinv_L\vc{r}_0}}$ (ISRV3);
if $\tU = \Pinvt_R\Pinv_R$,
   then the left system is constructed,
   because $\UL{\UL{\vc{r}^\SH_0=\Pinv\vc{r}_0}}$ (ISRV1);
and
if $\tU = \Pinvt_R\PO^\T\PO_L = \PO^\T_L\PO_L$,
   then the right system is constructed,
   because  $\UL{\UL{\vc{r}^\SH_0=\PO^\T\vc{r}_0}}$ (ISRV2).
   \qquad \endproof
\end{itemize}

In the next section,
\Thm~\ref{thm:dir_prec_by_isrv} will be verified numerically.

Here,
we note the following remarks; further information can be found in 
\App~\ref{appsec:bahavior_pol_pbicg}.

\begin{remark}
\label{rem:dir_ksp_and_isrv}
For any items for Case I, in the final structure,
the coefficient matrix in the residual polynomial
is the same as the direction of the preconditioned system;
further,
the initial residual vector of the linear system
is the same as
that of the shadow system; 
that is,
 $\tvc{r}^\SH_0=\tvc{r}_0$.

Specifically,
in the left system (ISRV1),
the final structure is
\begin{eqnarray}
&&
\L( \pR^\rmL_k(A^\T\Pinvt)(\Pinv\vc{r}_0),
 \pR^\rmL_k(\Pinv A)(\Pinv\vc{r}_0) \R); \NON
\end{eqnarray}
in the right system (ISRV2),
the final structure is
\begin{eqnarray}
&&
\L( \pR^\rmR_k(\Pinvt A^\T)\vc{r}_0,
 \pR^\rmR_k(A\Pinv)\vc{r}_0 \R); \NON
\end{eqnarray}
and in the two-sided system (ISRV3),
the final structure is
\begin{eqnarray}
&&
\L( \pR^\rmW_k(\Pinvt_R A^\T\Pinvt_L )(\Pinv_L\vc{r}_0),
 \pR^\rmW_k(\Pinv_L A\Pinv_R)(\Pinv_L\vc{r}_0) \R). \NON
\end{eqnarray}

\noindent
Note that here 
\Prop~\ref{prop:pbicg_biortho} is satisfied,
and
\Rem~{\it \ref{rem:std_pbicg_prec_oper}} (section ~\ref{ssec:improved1_pbicg}) and
\Rem~{\it \ref{rem:impr2_pbicg_prec_oper}} (section ~\ref{ssec:improved2_pbicg})
become apparent.
\end{remark}

\begin{remark}
\label{rem:arbitrary_isrv}
From \Prop~\ref{prop:bicr_by_isrv}
 and Case II in the proof of \Thm~\ref{thm:dir_prec_by_isrv},
if either $U$ or $\tU$ is arbitrarily chosen,
then the appropriate method for solving and 
the direction of the preconditioned system
may be indeterminable.
Even if $\vc{r}_0$ is replaced by an arbitrary vector $\vc{s}$,
then Case II
is still proven without loss of generality,
because $\vc{s} = U\vc{r}_0$.
Conversely,
if $U$ or $\tU$ is defined adequately,
as in Case I for the PBiCG,
then the appropriate solving method and the direction of the preconditioned system
can be determined.
\end{remark}

\begin{remark}
\label{rem:isrv_pbicg_for_sym}
As mentioned in section~\ref{ssec:improved1_pbicg},
there are instances in which
$(\vc{r}^\SH_0, \vc{r}_0)$ $\neq 0$ (e.g., $\vc{r}^\SH_0=\vc{r}_0$)
at the initial part of the standard PBiCG.
However,
in this case, $\tA$ in \Alg~\ref{alg:pbicg_simple} must be SPD,
with the modification
$(\tvc{r}^\SH_0, \tvc{r}_0) \neq 0$ (e.g., $\vc{r}^\SH_0=\vc{r}_0$).
The reason for this is as follows.
Let $A$ be SPD with the preconditioner $\PO=CC^{\rm T}$ ($\PO\approx A$),
then the two-sided preconditioning requires $\tA=C^{-1}AC^{-\T}$
in order to ensure it is still SPD,
and ISRV3 is constructed as $\vc{r}^\SH_0=CC^{-1}\vc{r}_0=\vc{r}_0$.
\end{remark}

\begin{remark}
\label{rem:dfn_of_prec_biLanczos}
The definition of the direction of a preconditioned system
for the BiCG method
requires \Thm~\ref{thm:dir_prec_by_isrv}
in addition to \Dfn~\ref{dfn:dir_prec}.
\end{remark}

%%%%%%  section  %%%%%%%%%%%%%%%%%%%%%%%%%%%%%%%%%%%%%%%%%%%%%%%%%%%%%%%%%%%%%%%%%
\section{Numerical experiments}
\label{sec:numerical_experiments}

In section~\ref{ssec:num_exp1},
by comparing the value of $\alpha_k$ and $\beta_k$ for each of
the four PBiCG algorithms presented in section~\ref{sec:pbicg_alg}
and
their corresponding PCGS algorithms \cite{itoh2019a},
we verify that the behavior of the right system is different from that of the
other preconditioned systems
(i.e., the left-preconditioned algorithms
and the improved preconditioned algorithms).
Next,
the switching of the direction of the preconditioned system
by the construction and setting of the ISRV
(\Thm~\ref{thm:dir_prec_by_isrv})
is verified in section~\ref{ssec:num_exp2}.

\subsection{Behavior of $\alpha_k$ and $\beta_k$ in the four PBiCG methods and
their corresponding PCGS methods}
\label{ssec:num_exp1}

The test problems were generated
by using real nonsymmetric matrices obtained from
 the Matrix Market \cite{matrixmarket}({\tt sherman4} and {\tt watt\UB \UB 1}).
The RHS vector $\vc{b}$ of (\ref{eqn:linear})
was generated by setting all elements of the exact solution vector
$\vc{x}_{\rm exact}$
to 1.0.
The initial solution was $\vc{x}_0 = \vc{0}$.

The numerical experiments were executed on a Dell Precision T7400
(Intel Xeon E5420, 2.5 GHz CPU, 16 GB RAM) running
the Cent OS (kernel 2.6.18)
and Matlab 7.8.0 (R2009a).

In all tests,
ILU(0) was adopted as the preconditioning operation,
and the value ``zero'' was set to mean the {\it fill-in} level.
The ISRVs were 
$\vc{r}^\flat_0 = \vc{r}_0$
 in the PBiCG corresponding to the conventional PCGS
 (\Alg~\ref{alg:pbicg_corresp_conv_pcgs})
and the conventional PCGS, they were 
$\vc{r}^\SH_0 = \vc{r}^+_0$
 in the left-PBiCG (\Alg~\ref{alg:pbicg_corresp_Llpcgs})
and the left-PCGS,
and they were
$\vc{r}^\SH_0 = \Pinv\vc{r}_0$
 in the standard PBiCG (\Alg~\ref{alg:pbicg}),
 the  PBiCG corresponding to Improved2 (\Alg~\ref{alg:improved2_pbicg}),
Improved1 (PCGS), and Improved2 (PCGS).

We plotted the values of $\alpha_k$ and $\beta_k$
for each of the four PBiCG algorithms presented in section~\ref{sec:pbicg_alg}
and for each of their corresponding PCGS algorithms \cite{itoh2019a};
these are shown in Figures \ref{fig:sherman4-alp1} to \ref{fig:watt__1-bet2}.

%%%%%%%%%%%%%%%%%%%%%%%%%%%%%%%%%%%%%%%%%%%%%%%%%%%%%%%%%%%%%%%%%%%%%%%%%

\def\RATIO1{0.4}
\def\CASE{sherman4}
\begin{figure}[htbp]
\begin{center}
\resizebox*{\FIGSIZE\columnwidth}{\RATIO1\height}{
\includegraphics*{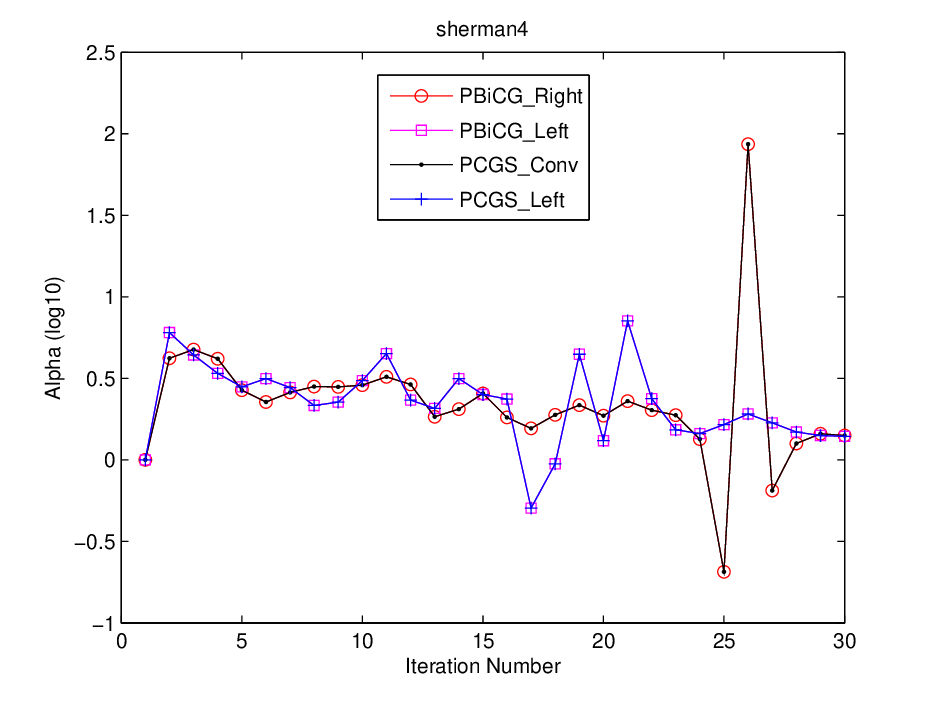}}
\caption{Values of $\alpha_k$ for the right- and left-PBiCG,
 and those of the corresponding PCGS methods
 (\CASE).}
\label{fig:sherman4-alp1}
\end{center}

\begin{center}
\resizebox*{\FIGSIZE\columnwidth}{\RATIO1\height}{
\includegraphics*{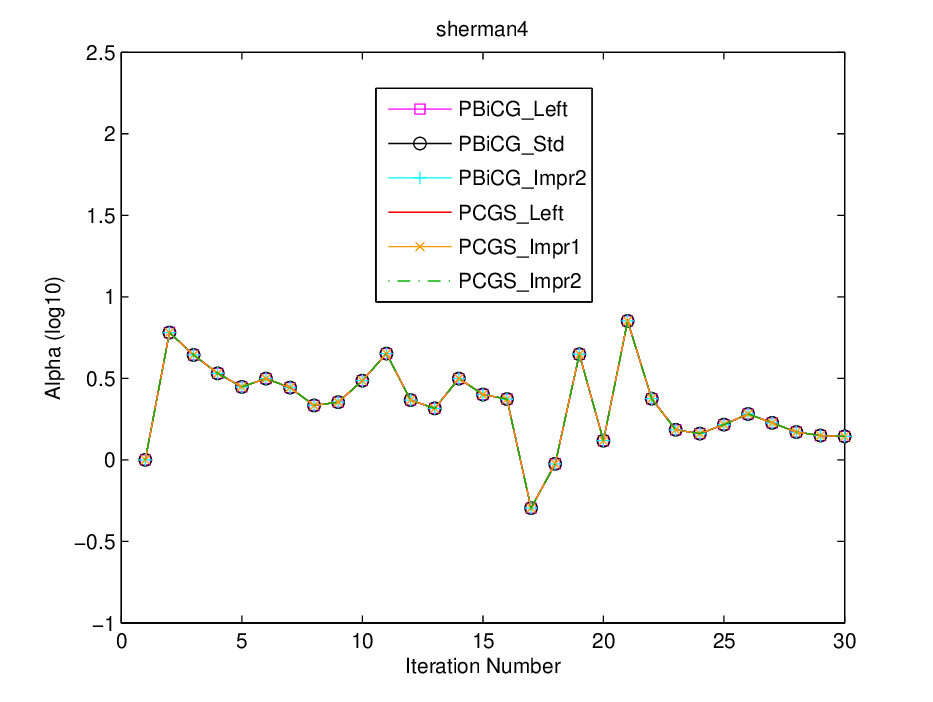}}
\caption{Value of $\alpha_k$ for the left- and standard PBiCG and the Improved2 PBiCG,
 and that of their corresponding PCGS methods
 (\CASE).}
\label{fig:sherman4-alp2}
\end{center}

\begin{center}
\resizebox*{\FIGSIZE\columnwidth}{\RATIO1\height}{
\includegraphics*{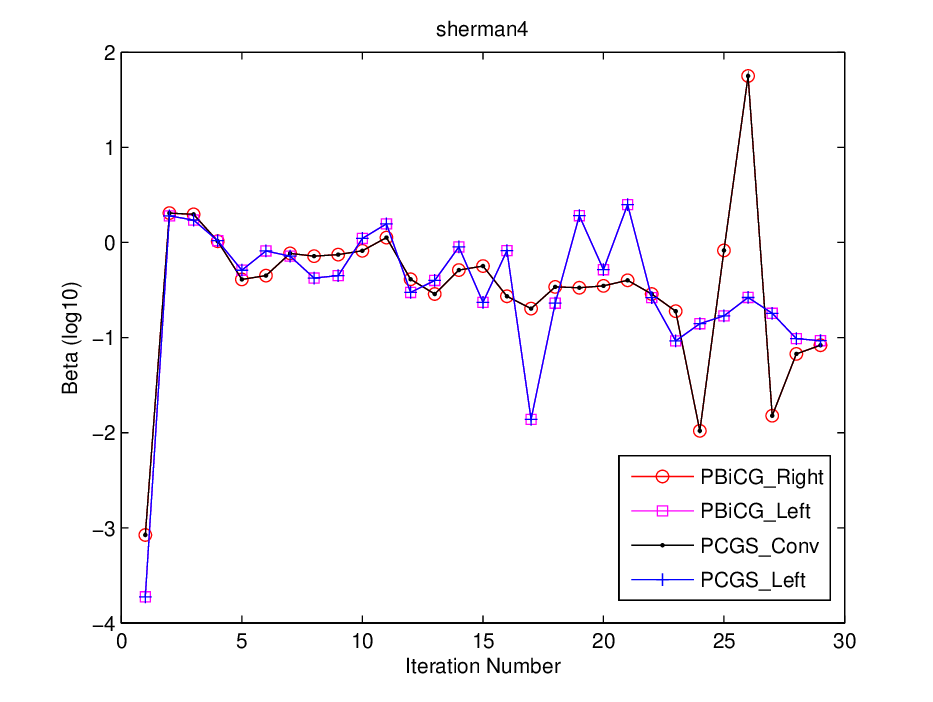}}
\caption{Value of $\beta_k$ for the right- and left-PBiCG,
 and that of the corresponding PCGS method
 (\CASE).}
\label{fig:sherman4-bet1}
\end{center}

\begin{center}
\resizebox*{\FIGSIZE\columnwidth}{\RATIO1\height}{
\includegraphics*{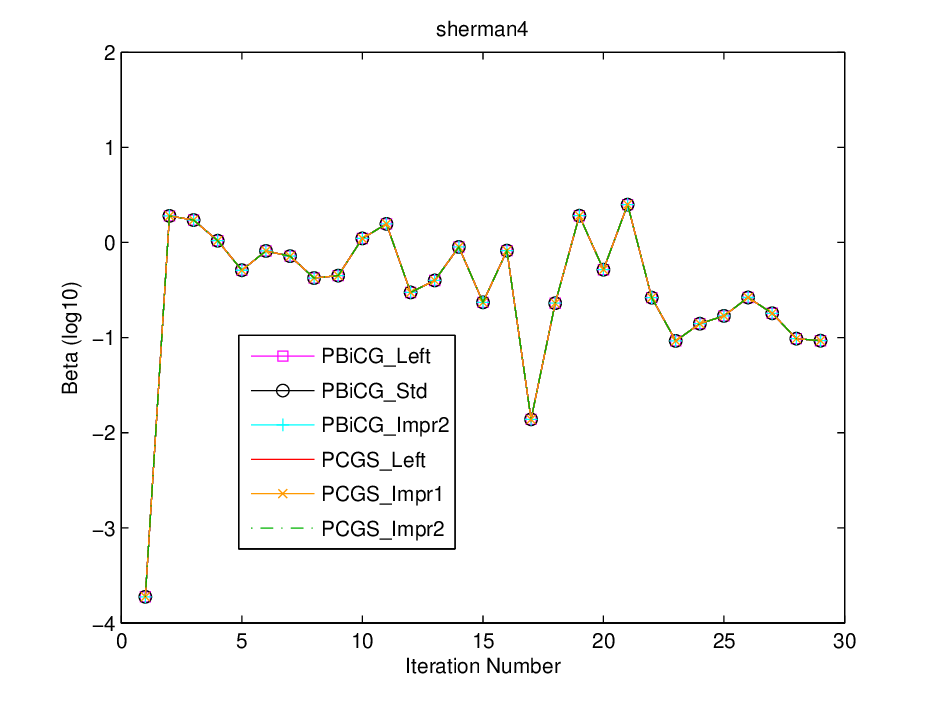}}
\caption{Value of $\beta_k$ for the left- and standard PBiCG and the Improved2 PBiCG,
 and that of the corresponding PCGS methods
 (\CASE).}
\label{fig:sherman4-bet2}
\end{center}
\end{figure}

\def\CASE{watt\UB \UB 1}
\begin{figure}[htbp]
\begin{center}
\resizebox*{\FIGSIZE\columnwidth}{\RATIO1\height}{
\includegraphics*{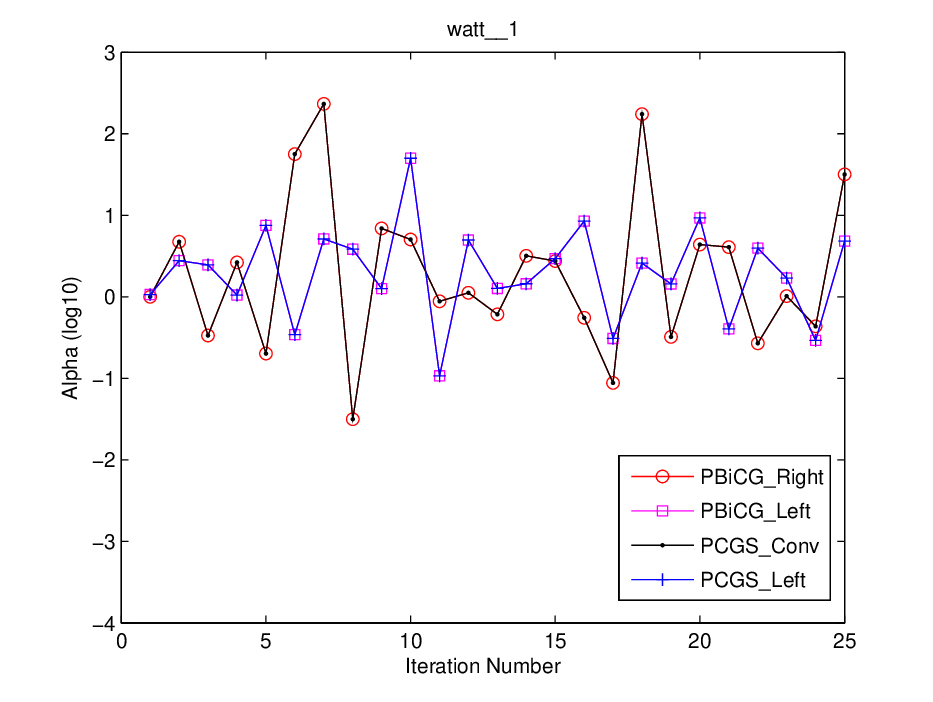}}
\caption{Value of $\alpha_k$ for the right- and left-PBiCG,
 and that of the corresponding PCGS method
 (\CASE).}
\label{fig:watt__1-alp1}
\end{center}

\begin{center}
\resizebox*{\FIGSIZE\columnwidth}{\RATIO1\height}{
\includegraphics*{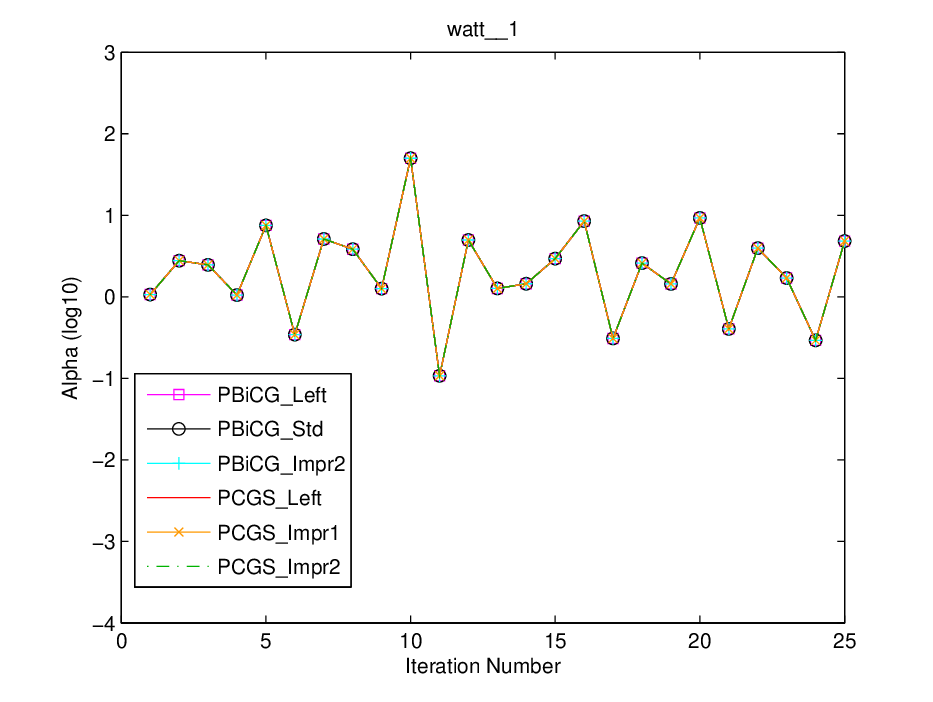}}
\caption{Value of $\alpha_k$ for the left- and standard PBiCG and the Improved2 PBiCG,
 and that of the corresponding PCGS method
 (\CASE).}
\label{fig:watt__1-alp2}
\end{center}

\begin{center}
\resizebox*{\FIGSIZE\columnwidth}{\RATIO1\height}{
\includegraphics*{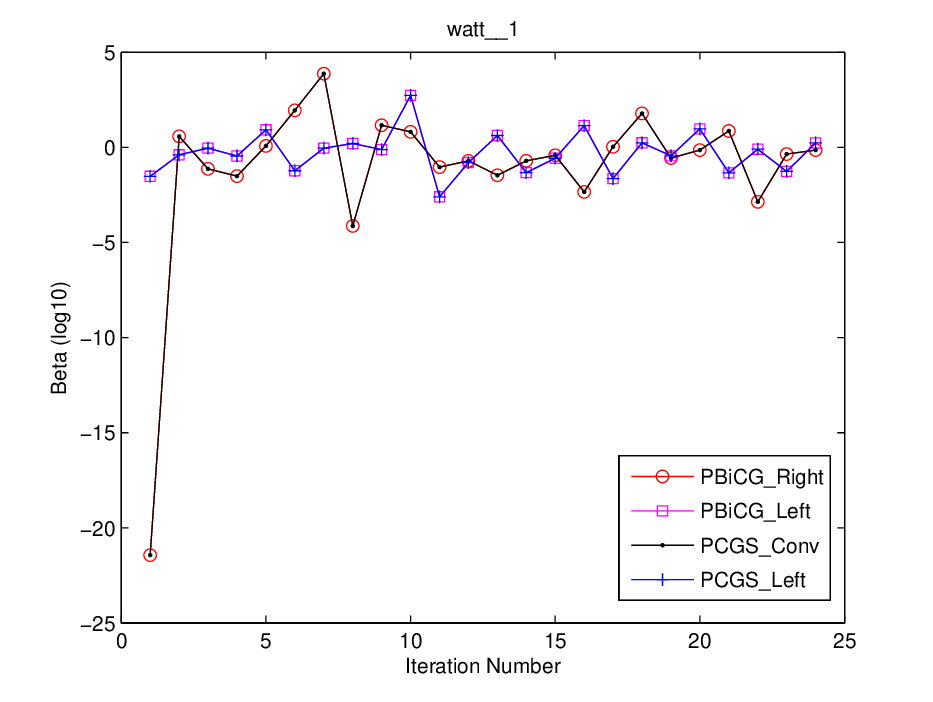}}
\caption{Value of $\beta_k$ for the right- and left-PBiCG,
 and that of the corresponding PCGS method
 (\CASE).}
\label{fig:watt__1-bet1}
\end{center}

\begin{center}
\resizebox*{\FIGSIZE\columnwidth}{\RATIO1\height}{
\includegraphics*{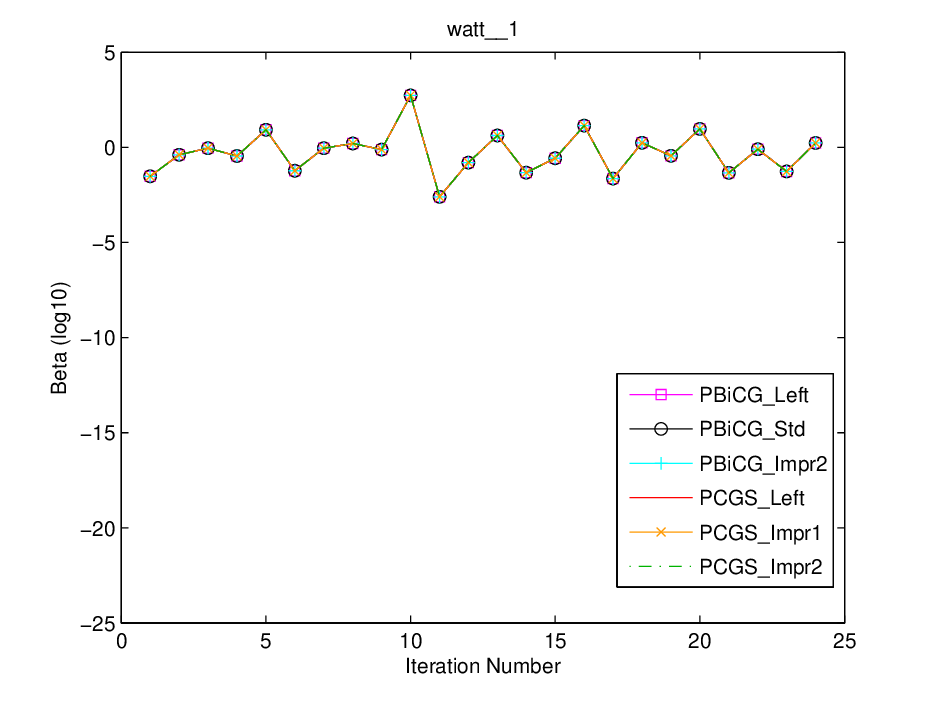}}
\caption{Value of $\beta_k$ for the left- and standard PBiCG and the Improved2 PBiCG,
 and that of the corresponding PCGS method
 (\CASE).}
\label{fig:watt__1-bet2}
\end{center}
\end{figure}
%

%%%%%%%%%%%%%%%%%%%%%%%%%%%%%%%%%%%%%%%%%%%%%%%%%%%%%%%%%%%%%%%%%%%%%%%%%

The labels in the graphs are as follows:

PBiCG\UB Right (\Alg~\ref{alg:pbicg_corresp_conv_pcgs})
 means the PBiCG corresponding to the conventional PCGS,
 that is, the right-preconditioned system.

PBiCG\UB Left (\Alg~\ref{alg:pbicg_corresp_Llpcgs})
 means the PBiCG of the left-preconditioned system.

PBiCG\UB Std (\Alg~\ref{alg:pbicg})
 means the PBiCG of the standard preconditioned BiCG,
 that is, the
PBiCG corresponding to Improved1.

PBiCG\UB Impr2 (\Alg~\ref{alg:improved2_pbicg})
 means the PBiCG corresponding to the Improved2 PCGS.

PCGS\UB Conv means the PCGS of the conventional preconditioning conversion.

PCGS\UB Left means the PCGS of the left-preconditioned system.

PCGS\UB Impr1 means the PCGS of Improved1.

PCGS\UB Impr2 means the PCGS of Improved2.
%

%%%%%%%%%%%%%%%%%%%%%%%%%%%%%%%%%%%%%%%%%%%%%%%%%%%%%%%%%%%%%%%%%%%%%%%%%

Figures ~\ref{fig:sherman4-alp1} and \ref{fig:watt__1-alp1} show
the behavior of $\alpha_k$ for the right-PBiCG
and the left-PBiCG and their corresponding PCGS algorithms.
Figures~\ref{fig:sherman4-alp2} and \ref{fig:watt__1-alp2} show
the behavior of $\alpha_k$ for the left-PBiCG, the standard PBiCG,
the Improved2 PBiCG (the PBiCG corresponding to the Improved2 PCGS),
and the corresponding PCGS algorithms.
From these results,
we know that for each of the four PBiCGs,
the value of $\alpha_k$ is the same as that in their respective PCGS,
but
the values for the right-PBiCG and for the conventional PCGS
are different from the others.
A comparison of these results on $\beta_k$ can be seen in Figures
 \ref{fig:sherman4-bet1}, \ref{fig:sherman4-bet2}, \ref{fig:watt__1-bet1},
and
\ref{fig:watt__1-bet2}.

%%%%%%%%%%%%%%%%%%%%%%%%%%%%%%%%%%%%%%%%%%%%%%%%%%%%%%%%%%%%%%%%%%%%%%%%%

In these graphs,
the behaviors of $\alpha_k$ and $\beta_k$ are the same
for each PBiCG algorithm and its corresponding PCGS algorithm;
that is,
we numerically verified the correspondence
between
 the PBiCG algorithms in \Fig~\ref{fig:PBiCG_ConvLlprecImproved}
 in section \ref{sec:pbicg_alg}
and
 the PCGS algorithms in \Fig~\ref{fig:PCGS_ConvLlprecImproved}
 (also see \cite{itoh2019a}).
We also verified that
the standard PBiCG (\Alg~\ref{alg:pbicg}) is coordinative to
the left-PBiCG (\Alg~\ref{alg:pbicg_corresp_Llpcgs});
that is,
$\alpha_k$ and $\beta_k$ are equivalent,
although
 the residual vector is not
($\vc{r}^+_k \equiv \Pinv\vc{r}_k$, where
$\vc{r}_k$ is the standard PBiCG, and $\vc{r}^+_k$ is the left-PBiCG).
We also verified the difference
between the right-preconditioned system
and the left-preconditioned system, including the standard PBiCG,
because
the behavior of $\alpha_k$ and $\beta_k$ in
the conventional PCGS and its corresponding PBiCG
are different from the behaviors seen in the other algorithms.

\subsection{Behavior of the left-, right-, and two-sided PBiCG
and standard PBiCG when switched by the ISRV}
\label{ssec:num_exp2}

For the experiments described in this subsection,
the experimental environment was the same as that described
 in section~\ref{ssec:num_exp1},
but the ISRVs of the PBiCG method were different.

We will verify \Thm~\ref{thm:dir_prec_by_isrv} by using the
BiCG for the preconditioned system (\Alg~\ref{alg:pbicg_simple})
 and
the standard PBiCG (\Alg~\ref{alg:pbicg}) with three different ISRVs.
Here,
\Alg~\ref{alg:pbicg_simple} is 
based on \Dfn~\ref{dfn:dir_prec},
and
\Alg~\ref{alg:pbicg_simple} is used to construct the left-preconditioned system
with $\PO_L = \PO$ and $\PO_R = I$ (PrecDirl-BiCG);
it is used to construct the right-preconditioned system
with $\PO_L = I$ and $\PO_R = \PO$ (PrecDirr-BiCG);
and
it is used to construct the two-sided preconditioned system
(PrecDirw-BiCG),
for the above algorithms;
the ISRV was uniformly set to $\tvc{r}^\SH_0 = \tvc{r}_0$.
The algorithm relative residual 2-norm was adjusted as following:
$\|\PO\tvc{r}_{k+1}\|_2/\|\vc{b}\|_2$ for the left system,
$\|\tvc{r}_{k+1}\|_2/\|\vc{b}\|_2$ for the right system,
and
$\|\PO_L\tvc{r}_{k+1}\|_2/\|\vc{b}\|_2$ for the two-sided system.
On the other hand,
as shown in \Thm~\ref{thm:dir_prec_by_isrv},
\Alg~\ref{alg:pbicg} was used to construct
the left-preconditioned system
with $\vc{r}^\SH_0 = \Pinv\vc{r}_0$ (ISRV1-PBiCG),
the right-preconditioned system
with $\vc{r}^\SH_0 = \PO^\T\vc{r}_0$ (ISRV2-PBiCG),
and
the two-sided preconditioned system
with $\vc{r}^\SH_0 = \PO^\T_R\Pinv_L\vc{r}_0$ (ISRV3-PBiCG),
these algorithm relative residual 2-norm were
all $\|\vc{r}_{k+1}\|_2/\|\vc{b}\|_2$ .

\def\RATIOTWO{0.4}
\def\CASE{sherman4}
\begin{figure}[ht]
\begin{center}
\resizebox*{\FIGSIZE\columnwidth}{\RATIOTWO\height}{
\includegraphics*{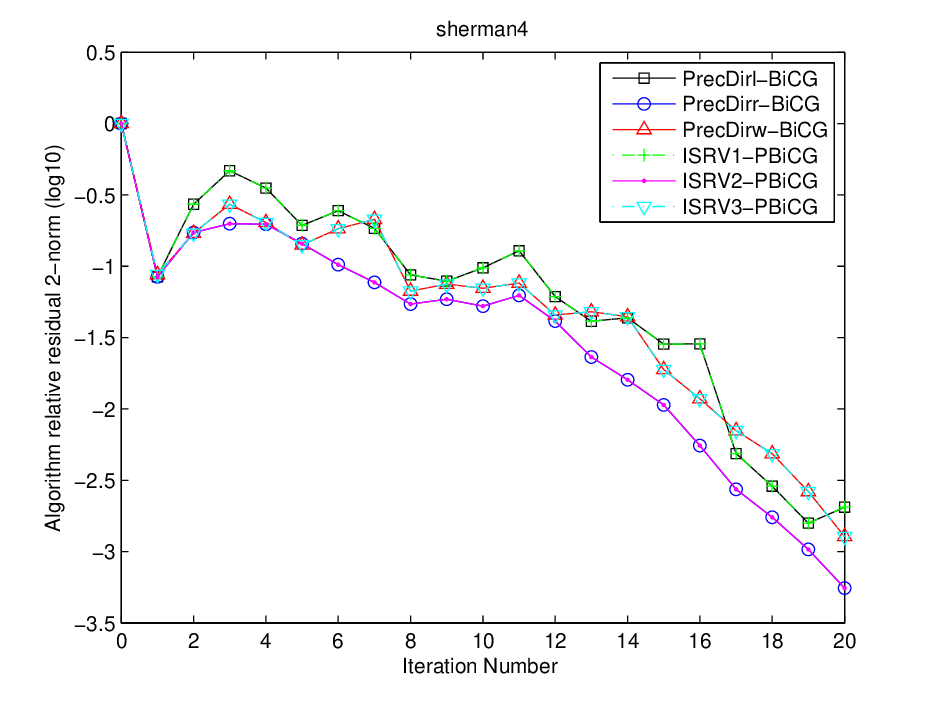}}
\caption{Behavior of the algorithm relative residual 2-norm for
 the left-, right-, and two-sided PBiCG and the
 standard PBiCG, with three different settings for the ISRV
 (\CASE).}
\label{fig:sherman4-isrv}
\end{center}

\def\CASE{watt\UB \UB 1}
\begin{center}
\resizebox*{\FIGSIZE\columnwidth}{\RATIOTWO\height}{
\includegraphics*{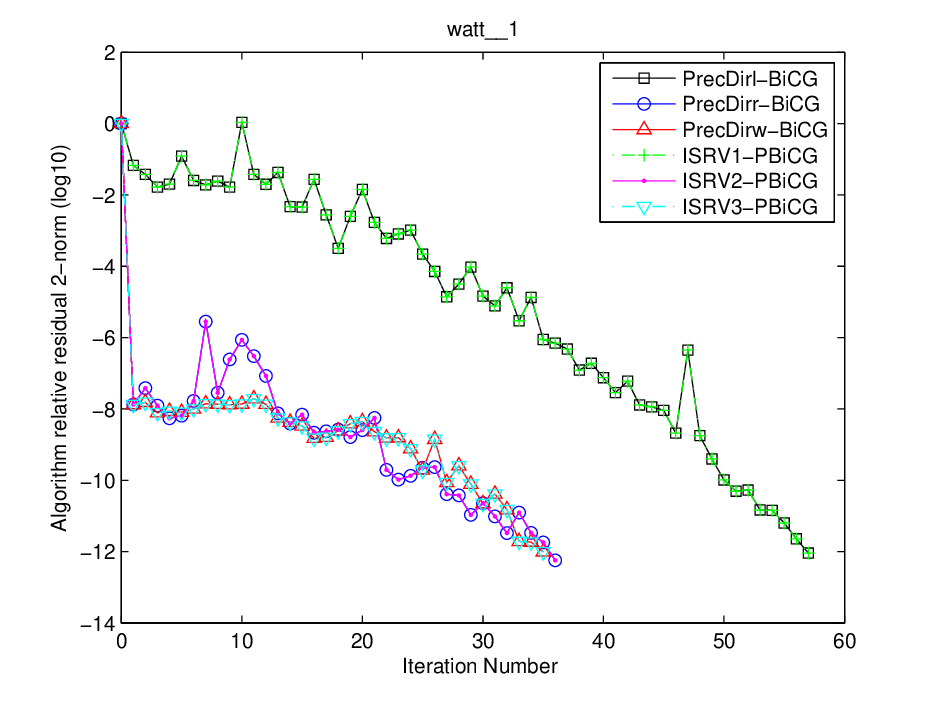}}
\caption{Behavior of the algorithm relative residual 2-norm for
 the left-, right-, and two-sided PBiCG and the
 standard PBiCG, with three different settings for the ISRV
 (\CASE).}
\label{fig:watt__1-isrv}
\end{center}
\end{figure}

Figures \ref{fig:sherman4-isrv} and \ref{fig:watt__1-isrv} illustrate
the equivalence of the direction of a preconditioned system obtained
by \Alg~\ref{alg:pbicg_simple} based on \Dfn~\ref{dfn:dir_prec}
 and
the direction switching due to the ISRV when using \Alg~\ref{alg:pbicg}; this occurs
because
 the left-preconditioned system (PrecDirl-BiCG) has the same behavior
 as that of the standard PBiCG with ISRV1, %($\vc{r}^\SH_0=\Pinv\vc{r}_0$)
the right-preconditioned system (PrecDirr-BiCG) has the same behavior
 as that of the standard PBiCG with ISRV2, %($\vc{r}^\SH_0=\PO^\T\vc{r}_0$)
and
 the two-sided preconditioned system (PrecDirw-BiCG) has the same behavior
 as that of the standard PBiCG with ISRV3. %($\vc{r}^\SH_0=\PO^\T_R\Pinv_L\vc{r}_0$)

%%%%%%  section  %%%%%%%%%%%%%%%%%%%%%%%%%%%%%%%%%%%%%%%%%%%%%%%%%%%%%%%%%%%%%%%%%
\section{Conclusions}
\label{sec:conclusion}

In this paper,
we analyzed four different preconditioned BiCG (PBiCG) algorithms, 
from the viewpoint of their polynomial structure.
These PBiCG algorithms correspond to the four PCGS algorithms considered
in \cite{itoh2019a}.

We have shown the mechanism that determines the direction
of such a preconditioned system; that is, the
direction %of preconditioned system of PBiCG
is determined by $\alpha_k$ and $\beta_k$,
which are constructed by biorthogonal and biconjugate operations.
However,
the biorthogonal and biconjugate structures of the polynomials
of 
the four PBiCG methods are all the same.
Therefore,
we have identified that the final factor that can switch the direction
of such a preconditioned system is the
construction and setting
of the ISRV.
In particular,
we have shown that
the direction of the preconditioned system 
has never been fixed
without using the relation $\tvc{r}^\SH_0 = \tvc{r}_0$.
Furthermore,
we have shown an additional theorem
regarding the definition of the direction of a preconditioned system
for a BiCG method for solving linear equations.
In other words,
the construction and setting of the ISRV affect 
not only the shadow system,
but also the linear system
on the direction of the preconditioned system,
due to the inner product of $\alpha_k$ and $\beta_k$.

These properties of PBiCG methods are commonly discussed
in the literature of preconditioned bi-Lanczos-type algorithms,
for example,
preconditioned CGS (PCGS)
 and
preconditioned BiCG stabilized (PBiCGStab) algorithms \cite{vorst1992},
and so on.
Further,
\Thm~\ref{thm:dir_prec_by_isrv} presented by this paper
is able to be extended
to a variety of preconditioned bi-Lanczos-type algorithms.
On the other hand,
PCGS algorithms are congruent to the direction of the preconditioning conversion,
and this has already been analyzed \cite{itoh2019a};
though,
PBiCGStab algorithms are not congruent,
and they will be analyzed as an area of future work.

%%%%%%%%%%%%%%%%%%%%%%%%%%%%%%%%%%%%%%%%%%%%%%%%%%%%%%%%%%%%%%%%%%%%%%%%%%%%%%%%%%

%%%%%%%%%%%%%%%%%%%%%%%%%%%%%%%%%%%%%%%%%%%%%%%%%%%%%%%%%%%%%%%%%%%%%%%%%%%%%%%%%%
\appendix
\section{Stepwise analysis of the polynomials of the standard PBiCG}
\label{appsec:bahavior_pol_pbicg}

% in one instance,
Here we present detailed examples of
the polynomials of the standard PBiCG (\Alg~\ref{alg:pbicg})
when using 
ISRV3 ($\vc{r}^\SH_0 = \PO^\T_R\Pinv_L\vc{r}_0$, Example 1)
and
 the ISRV1 ($\vc{r}^\SH_0 = \Pinv\vc{r}_0$, Example 2); we perform a
stepwise analysis
by using the recurrence relations
 (\ref{eqn:prec_init_pol}) to (\ref{eqn:prec_prob_pol}) in \Sec~\ref{sec:pbicg_alg}.

We will use the following notation:
$\tA_w$ ( $= \Pinv_L A \Pinv_R$) means the two-sided preconditioning conversion,
$\tA_l$ ( $= \Pinv A$) means the left preconditioning conversion,
 and 
$\tA_r$ ( $= A\Pinv$) means the right preconditioning conversion.

The initial values of the polynomials in the preconditioned system
are as follows:
\begin{eqnarray}
\pP_0(\tA_w)=
\pP_0(\tA_l)=
\pP_0(\tA_r)=I, %\NON
\label{eqn:pol_prob_zero}
\\
\pR_0(\tA_w)=
\pR_0(\tA_l)=
\pR_0(\tA_r)=I. %\NON
\label{eqn:pol_res_zero}
\end{eqnarray}

% \newpage
\def\EQUIV{=}

\def\At{A^\T}
\def\PR{\PO_R}
\def\PL{\PO_L}
\def\Pt{\PO^\T}
\def\PRt{\PR^\T}
\def\PLt{\PL^\T}
\def\PRinv{\Pinv_R}
\def\PLinv{\Pinv_L}
\def\PRinvt{\Pinvt_R}
\def\PLinvt{\Pinvt_L}

\def\r0{\vc{r}_0}
\def\isrv3{(\PRt\PLinv)\r0}

\def\palpha{\alpha}
\def\pbeta{\beta}

%%%%%%%%%%%%%%%%%%%%%%%%%%%%%%%%%%%%%%%%%%%%%%%%%%%%%%%%%%%%%
%\stepcounter{alg}
\noindent
% \hspace{\ALGTITLEWIDTH}
{\bf Example 1. Details of standard PBiCG algorithm with ISRV3:}
\begin{indention}{0em}
%%%%%%%%%%%%%%%%%%%%%%%%%%%%%%%%%%%%%%%%%%%%%%%%%%%%%%%%%%%%%
\noindent
$\vc{x}_0$ is an initial guess,
$\SP \vc{r}_0= \vc{b}-A\vc{x}_0, \SP\SP$
set $\SP \beta_{-1}=0$,
\\
$\LA \tvc{r}^\SH_0, \tvc{r}_0 \RA
 =
 \LA \PRinvt\vc{r}^\SH_0, \PLinv\vc{r}_0 \RA
 \neq 0$,
e.g.,
$\vc{r}^\SH_0 = \PRt\PLinv\vc{r}_0$, %\SP \\
%
%
%%%%%%%%
\def\k{0}
\def\kp{1}
\def\km{{-1}}
\begin{eqnarray}
k = \k \SP : && \NON
\\
%%%%%% both %%%%%%
\vc{p}^+_\k     &=& \Pinv \vc{r}_\k = \Pinv_R\Pinv_L\r0, %+ \beta_\km\vc{p}_\km
\label{eqn:both_k=\k_p3}
\\
\vc{p}^\flat_\k &=& \Pinvt\vc{r}^\SH_\k %+ \beta_\km\vc{p}^\SH_\km 
               =  \Pinvt\isrv3  % \NON  % + \beta_\km\vc{p}^\SH_\km \NON
               =  \PLinvt\PLinv\r0 ,  % + \beta_\km\vc{p}^\SH_\km %, \NON
\label{eqn:both_k=\k_pS}
\\
\palpha_\k &=& \frac
            {\LA \vc{r}^\SH_\k,   \Pinv\vc{r}_\k \RA}
            {\LA \vc{p}^\flat_\k, A \vc{p}^+_\k      \RA} %\NON 
           =  \frac
            {\LA \isrv3           ,   \Pinv\r0 \RA}
            {\LA \PLinvt\PLinv\r0, A (\Pinv\r0) \RA}  % \NON
\\
          &=& \frac
            {\LA \PLinv\r0,                    \PLinv \r0 \RA}
            {\LA \PLinv\r0, (\PLinv A \PRinv ) \PLinv \r0 \RA} %, \NON
           \equiv \palpha^\rmW_\k , \NON
\label{eqn:both_k=\k_alp1}
\\
\vc{x}_\kp &=& \vc{x}_\k + \palpha^\rmW_\k\vc{p}^+_\k  %\NON
            =  \vc{x}_\k + \palpha^\rmW_\k\Pinv  \r0   %\NON
            =  \vc{x}_\k + \palpha^\rmW_\k\Pinv_R\Pinv_L  \r0,   %\NON
\label{eqn:both_k=\k_x}
\\
\vc{r}_\kp &=& \vc{r}_\k - \palpha^\rmW_\k A\vc{p}^+_\k \NON
\\
%%%%%% both %%%%%%
           &=& \PL (I - \palpha^\rmW_\k (\PLinv A \PRinv ) ) \PLinv\r0
\SP\SP\SP\SP\SP\SP\SP\SP\SP\SP\SP\SP\SP
           \EQUIV \UL{\UL{\PL \pR^\rmW_\kp(\tA_w) \PLinv\r0}}
%   \NON
% \\
\label{eqn:both_k=\k_r3}
\\
%%%%%% left %%%%%%
           &=&  \P(I - \palpha^\rmW_\k(\Pinv A) ) \Pinv\r0 %\NON
\SP\SP\SP\SP\SP\SP\SP\SP\SP\SP\SP\SP\SP\SP\SP\SP\SP\SP\SP\SP\SP\SP\SP
           \EQUIV \UL{\UL{\P \pR^\rmW_\kp(\tA_l) \Pinv\r0}}
%  \NON
% \\
\label{eqn:both_k=\k_r1}
\\
%%%%%% right %%%%%%
           &=& \vc{r}_\k - \palpha^\rmW_\k A\Pinv \r0  %\NON
            =      (I - \palpha^\rmW_\k (A\Pinv) ) \r0 %\NON
\SP\SP\SP
           \EQUIV \UL{\UL{\pR^\rmW_\kp(\tA_r) \r0}} , %\NON
\label{eqn:both_k=\k_r2}
\\
\vc{r}^\SH_\kp &=& \vc{r}^\SH_\k - \palpha^\rmW_\k \At \vc{p}^\flat_\k \NON
                =  \isrv3 - \palpha^\rmW_\k \At (\PLinvt  \PLinv\r0) \NON
\\
%%%%%% both %%%%%%
               &=& \PRt(I - \palpha^\rmW_\k (\PRinvt\At \PLinvt) )\PLinv\r0
\SP\SP\SP\SP\SP\SP\SP\SP\SP\SP
               \EQUIV \UL{\UL{\PRt \pR^\rmW_\kp(\tA^\T_w) \PLinv\r0}}
%  \NON
% \\
\label{eqn:both_k=\k_rS3}
\\
%%%%%% left %%%%%%
               &=&  (I - \palpha^\rmW_\k (\At \Pinvt) )\isrv3  \NON
\\
               &\EQUIV& \pR^\rmW_\kp(\At \Pinvt) \isrv3 %, \NON
\SP\SP\SP\SP\SP\SP\SP\SP\SP\SP\SP\SP
\SP\SP\SP\SP\SP\SP\SP\SP\SP\SP\SP\SP
               \EQUIV \UL{\UL{\PRt \pR^\rmW_\kp(\tA^\T_w) \PLinv\r0}}
%  \NON
% \\
\label{eqn:both_k=\k_rS1}
\\
%%%%%% right %%%%%%
               &=& \Pt(I - \alpha^\rmW_\k (\Pinvt \At) )\PLinvt\PLinv\r0  \NON %
\\
               &\EQUIV& \Pt\pR^\rmW_\kp(\Pinvt\At)\PLinvt\PLinv\r0 %, \NON
\SP\SP\SP\SP\SP\SP\SP\SP\SP\SP\SP\SP
\SP\SP\SP
\SP\SP\SP
               \EQUIV \UL{\UL{\PRt \pR^\rmW_\kp(\tA^\T_w) \PLinv\r0}} ,
%  \NON
% \\
\label{eqn:both_k=\k_rS2}
\\
\pbeta_\k &=& \frac
            {\LA \vc{r}^\SH_\kp, \Pinv\vc{r}_\kp \RA}
            {\LA \vc{r}^\SH_\k,  \Pinv\vc{r}_\k  \RA} %\NON
          =  \frac
            {\LA \PRt \pR^\rmW_\kp(\tA^\T_w) \PLinv\r0, \Pinv\PL \pR^\rmW_\kp(\tA_w) \PLinv\r0 \RA}
            {\LA \pR_\k(\tA^\T_w) (\PRt \PLinv\r0) ,  \Pinv \pR_\k(\tA_w) \r0  \RA} %\NON
\\
         &=& \frac
            {\LA \pR^\rmW_\kp(\tA^\T_w) \PLinv\r0 , \pR^\rmW_\kp(\tA_w) \PLinv\r0 \RA}
            {\LA \pR_\k(\tA^\T_w)  \PLinv\r0 , \pR_\k(\tA_w)  \PLinv\r0 \RA}
          \equiv \pbeta^\rmW_\k . \NON
\label{eqn:both_k=\k_bet1}
\end{eqnarray}

The double-underlined equations show the important polynomial structures.
% of note.
By way of contrast,
neither (\ref{eqn:both_k=0_p3}) nor (\ref{eqn:both_k=0_pS})
is double underlined, and
their polynomials are not displayed; this is
because they are the identity matrix,
as indicated in (\ref{eqn:pol_prob_zero}) and (\ref{eqn:pol_res_zero}).

In the above description,
we will focus on $\Pinv_L\r0$
in the final structure of each equation.
Because
$\Pinv_L\r0$ is the initial residual vector
of the two-sided preconditioned system,
details of its properties can be found in \Thm~\ref{thm:dir_prec_by_isrv}
and \Rem~{\it \ref{rem:dir_ksp_and_isrv}}
in section~\ref{sec:direction_preconditioning}.
However,
at steps (\ref{eqn:both_k=0_p3}) and (\ref{eqn:both_k=0_pS}),
the intrinsic structure of $\vc{p}^+_0$ and $\vc{p}^\flat_0$ does not
play a role in determining the direction of the preconditioned system,
% because neither vector contributes to $\alpha_0$ or $\beta_0$.
because neither vector has parameter $\alpha_0$ or $\beta_0$.

The direction of preconditioned system is thus fixed as the two-sided system
when $\alpha_0$ is calculated in  (\ref{eqn:both_k=0_alp1}).
The approximate solution vector $\vc{x}_1$ is calculated
for the two-sided system in (\ref{eqn:both_k=0_x}),
because (\ref{eqn:both_k=0_x}) has $\alpha^\rmW_0$.

The intrinsic structure of the residual vector $\vc{r}_1$ may be that of 
(\ref{eqn:both_k=0_r3}) to (\ref{eqn:both_k=0_r2}), that is,
two-sided, left, or right, respectively%
\footnote{
For the same reason,
$\vc{p}^+_0$ of (\ref{eqn:both_k=0_p3}) and $\vc{x}_1$ of (\ref{eqn:both_k=0_x})
may be
two-sided, left, or right.
}%
.
However,
the direction of the preconditioned system has been already fixed
in (\ref{eqn:both_k=0_alp1}), the operation on $\alpha_0$,
therefore,
the intrinsic structure of $\vc{r}_1$ is fixed as
$\vc{r}_1 = \PL (I - \palpha^\rmW_0 (\PLinv A \PRinv ) ) \PLinv\r0
          = \PL \pR^\rmW_1(\tA_w) \PLinv\r0$.
Furthermore,
this initial residual vector part is $\Pinv_L\r0$.

On the other hand, the
intrinsic structure of the residual vector $\vc{r}^\SH_1$ may be created by
(\ref{eqn:both_k=0_rS3}) to (\ref{eqn:both_k=0_rS2}),
but
these all reduce to the same structure,

\noindent
$\vc{r}^\SH_1 = \PRt(I - \palpha^\rmW_0 (\PRinvt\At \PLinvt) )\PLinv\r0
              = \PRt \pR^\rmW_1(\tA^\T_w) \PLinv\r0$.
The reason for this is that
the direction of the preconditioned system has been already fixed as $\alpha^\rmW_0$, the
same as for $\vc{r}_1$.
Furthermore,
the part of $\Pinv_L\r0$
and the shadow system with the transpose matrices may not be compatible%
\footnote{
For the same reason
as for $\vc{p}^\flat_0$ of (\ref{eqn:both_k=0_pS}),
the part of $\Pinv_L\r0$
and the shadow system with the transpose matrices may not be compatible.
}.

When $\beta_0$ operates in the denominator,
$\pR_\k(\tA^\T_w)$ does not fix the direction of the preconditioned system
because of (\ref{eqn:pol_res_zero}).

The subsequent iterated operations are as follows:

%
%%%%%%%%
\def\k{k}
\def\kp{{k+1}}
\def\km{{k-1}}
{\rm For } $\SP k = 1, 2, 3, \cdots, {\rm Do:}$
\begin{eqnarray}
\vc{p}^+_\k     &=& \Pinv \vc{r}_\k + \pbeta^\rmW_\km\vc{p}^+_\km %\NON
%%%%%% both %%%%%%
              \EQUIV \UL{\UL{\PRinv \pP^\rmW_\k(\tA_w)\PLinv\r0}} ,
\NON
\label{eqn:both_k=\k_p3}
\\
\vc{p}^\flat_\k &=& \Pinvt\vc{r}^\SH_\k + \pbeta^\rmW_\km\vc{p}^\flat_\km %\NON
\NON
\label{eqn:both_k=\k_pS}
              \EQUIV \UL{\UL{\PLinvt \pP^\rmW_\k(\tA^\T_w) \PLinv\r0}}, %\NON
\\
\palpha^\rmW_\k &=& \frac
            {\LA\vc{r}^\SH_\k,   \Pinv\vc{r}_\k \RA}
            {\LA\vc{p}^\flat_\k, A \vc{p}^+_\k      \RA} % \NON 
\NON
\label{eqn:both_k=\k_alp1}
           =  \frac
            {\LA \pR^\rmW_\k(\tA^\T_w)\PLinv\r0, \pR^\rmW_\k(\tA_w)\PLinv\r0 \RA}
            {\LA \pP^\rmW_\k(\tA^\T_w)\PLinv\r0, (\PLinv A\PRinv)\pP^\rmW_\k(\tA_w)\PLinv\r0 \RA}
 , %\NON
\\
\vc{x}_\kp &=& \vc{x}_\k + \palpha^\rmW_\k\vc{p}^+_\k  %\NON
           \EQUIV \vc{x}_\k + \UL{\UL{\palpha^\rmW_\k\PRinv \pP^\rmW_\k(\tA_w)\PLinv \r0}}
, \NON
\\
\vc{r}_\kp &=& \vc{r}_\k - \palpha^\rmW_\k A\vc{p}^+_\k %\NON
%%%%%% both %%%%%%
\NON
\label{eqn:both_k=\k_r3}
           \EQUIV \UL{\UL{\PL \pR^\rmW_\kp(\tA_w) \PLinv\r0}} , %\NON
\\
\vc{r}^\SH_\kp &=& \vc{r}^\SH_\k - \palpha^\rmW_\k \At \vc{p}^\flat_\k %\NON
\NON
\label{eqn:both_k=\k_rS}
              \EQUIV \UL{\UL{\PRt \pR^\rmW_\kp(\tA^\T_w) \PLinv\r0}} , %\NON
\\
\pbeta^\rmW_\k &=& \frac
            {\LA \vc{r}^\SH_\kp, \Pinv\vc{r}_\kp \RA}
            {\LA \vc{r}^\SH_\k,  \Pinv\vc{r}_\k  \RA} %\NON
          =  \frac
            {\LA \pR^\rmW_\kp(\tA^\T_w)\PLinv\r0 , \pR^\rmW_\kp(\tA_w) \PLinv\r0 \RA}
            {\LA \pR^\rmW_\k(\tA^\T_w) \PLinv\r0 , \pR^\rmW_\k(\tA_w)\PLinv\r0 \RA} , %\NON
\NON
\end{eqnarray}
\end{indention}
End Do
\\

Next,
we will also briefly describe the
polynomials of the standard PBiCG (\Alg~\ref{alg:pbicg})
with ISRV1 ($\vc{r}^\SH_0 = \Pinv\vc{r}_0$).
The initial values of the polynomials for the left-preconditioned system
are
% \begin{eqnarray}
$\pP^\rmL_0(\tA_l) = \pR^\rmL_0(\tA_l)=I$.
% \end{eqnarray}

Refer to Example 1 for a detailed description.

% \newpage
%%%%%%%%%%%%%%%%%%%%%%%%%%%%%%%%%%%%%%%%%%%%%%%%%%%%%%%%%%%%%
%\stepcounter{alg}
\noindent
% \hspace{\ALGTITLEWIDTH}
{\bf Example 2. Polynomial description of the standard PBiCG algorithm with ISRV1:}
\begin{indention}{2em}
%%%%%%%%%%%%%%%%%%%%%%%%%%%%%%%%%%%%%%%%%%%%%%%%%%%%%%%%%%%%%
\noindent
$\vc{x}_0$ is an initial guess,
$\SP \vc{r}_0= \vc{b}-A\vc{x}_0, \SP\SP$
set $\SP \beta^\rmL_{-1}=0$,
\\
$\LA \tvc{r}^\SH_0, \tvc{r}_0 \RA
 =
 \LA \PRinvt\vc{r}^\SH_0, \PLinv\vc{r}_0 \RA
 \neq 0$,
e.g.,
$\vc{r}^\SH_0 = \Pinv\vc{r}_0$, %\SP
\\
{\rm For } $\SP k = 0, 1, 2, 3, \cdots,$ {\rm Do:}
%
%
%%%%%%%%
\def\k{k}
\def\kp{{k+1}}
\def\km{{k-1}}
\begin{eqnarray}
\vc{p}^+_\k     &=& \Pinv \vc{r}_\k + \pbeta^\rmL_\km\vc{p}^+_\km %\NON
%%%%%% left %%%%%%
              \EQUIV \UL{\UL{\pP^\rmL_\k(\tA_l)\Pinv\r0}} ,  \NON
\label{eqn:left_k=\k_p3}
\\
\vc{p}^\flat_\k &=& \Pinvt\vc{r}^\SH_\k + \pbeta^\rmL_\km\vc{p}^\flat_\km %\NON
\NON
\label{eqn:left_k=\k_pS}
              \EQUIV \UL{\UL{\Pinvt \pP^\rmL_\k(\tA^\T_l) \Pinv\r0}}, %\NON
\\
\palpha^\rmL_\k &=& \frac
            {\LA\vc{r}^\SH_\k,   \Pinv\vc{r}_\k \RA}
            {\LA\vc{p}^\flat_\k, A \vc{p}^+_\k      \RA} % \NON 
\NON
\label{eqn:left_k=\k_alp1}
           =  \frac
            {\LA \pR^\rmL_\k(\tA^\T_l)\Pinv\r0, \pR^\rmL_\k(\tA_l)\Pinv\r0 \RA}
            {\LA \pP^\rmL_\k(\tA^\T_l)\Pinv\r0, (\Pinv A)\pP^\rmL_\k(\tA_l)\Pinv\r0 \RA}
 , %\NON
\\
\vc{x}_\kp &=& \vc{x}_\k + \palpha^\rmL_\k\vc{p}^+_\k  %\NON
           \EQUIV \vc{x}_\k + \UL{\UL{\palpha^\rmL_\k \pP^\rmL_\k(\tA_l)\Pinv \r0}}
, \NON
\\
\vc{r}_\kp &=& \vc{r}_\k - \palpha^\rmL_\k A\vc{p}^+_\k %\NON
%%%%%% left %%%%%%
% \NON
\label{eqn:left_k=\k_r3}
           \EQUIV \UL{\UL{\PO \pR^\rmL_\kp(\tA_l) \Pinv\r0}} , %\NON
\\
\vc{r}^\SH_\kp &=& \vc{r}^\SH_\k - \palpha^\rmL_\k \At \vc{p}^\flat_\k %\NON
\NON
\label{eqn:left_k=\k_rS}
              \EQUIV \UL{\UL{\pR^\rmL_\kp(\tA^\T_l) \Pinv\r0}} , %\NON
\\
\pbeta^\rmL_\k &=& \frac
            {\LA \vc{r}^\SH_\kp, \Pinv\vc{r}_\kp \RA}
            {\LA \vc{r}^\SH_\k,  \Pinv\vc{r}_\k  \RA} %\NON
          =  \frac
            {\LA \pR^\rmL_\kp(\tA^\T_l)\Pinv\r0 , \pR^\rmL_\kp(\tA_l) \Pinv\r0 \RA}
            {\LA \pR^\rmL_\k(\tA^\T_l) \Pinv\r0 , \pR^\rmL_\k(\tA_l)\Pinv\r0 \RA} , %\NON
\NON
\end{eqnarray}
End Do
\\
\end{indention}

For
the polynomial structures of
% (\ref{eqn:left_k=k_p3}) and 
(\ref{eqn:left_k=k_r3}),
refer to
 \Rem~{\it \ref{rem:std_pbicg_prec_pol}} in section~\ref{ssec:improved1_pbicg}. %,
%  \Rem~{\it \ref{rem:impr2_pbicg_prec_pol}} in section~\ref{ssec:improved2_pbicg},
% and
%  Case I-1) in the proof of \Thm~\ref{thm:dir_prec_by_isrv}.

%%%% Acknowledgments %%%%%%%%
% \section*{Acknowledgments}
% This work was partially supported by JSPS KAKENHI Grant Number JP25390145.

\if0

%%%%%%%%%%%%%%%%%%%%%%%%%%%%%%%%%%%%%%%%%%%%%%%%%%%%%%%%%%%%%%%%%%%%%%%%%%%%%%%%%%

\fi

\end{document}